\documentclass[final,leqno,onefignum,onetabnum]{siamltex}

\usepackage[ansinew]{inputenc}
\usepackage[T1]{fontenc}
\usepackage{graphicx}
\usepackage{amsmath}
\usepackage{amssymb}
\usepackage{comment}
\usepackage{color}
\usepackage{epstopdf}
\usepackage{float}
\usepackage{mathrsfs}
\usepackage{hyperref}

\newtheorem{claim}{\bf Claim}

\newtheorem{remark}[theorem]{\emph{Remark}}

\title{Global stabilization of a Korteweg-de Vries equation with saturating distributed control\thanks{This work has been partially supported by  
Fondecyt 1140741, MathAmsud COSIP, and Basal Project FB0008 AC3E.}} 

\author{
Swann Marx\footnotemark[2] \and
Eduardo Cerpa\footnotemark[3] \and
 Christophe Prieur\footnotemark[2]
 \and Vincent Andrieu\footnotemark[4]}

\usepackage{amsfonts,latexsym,amsmath,amssymb}
\usepackage[mathscr]{eucal}
\usepackage{graphicx,color}
\usepackage{comment}
\usepackage{hyperref}
\catcode`\@=11
\def\downparenfill{$\m@th\braceld\leaders\vrule\hfill\bracerd$}
\def\overparen#1{\mathop{\vbox{\ialign{##\crcr\crcr
\noalign{\kern0.4ex}
\downparenfill\crcr\noalign{\kern0.4ex\nointerlineskip}
$\hfil\displaystyle{#1}\hfil$\crcr}}}\limits}
\catcode`\@=12

\def\sath{\mathfrak{sat}_2}
\def\satl{\mathfrak{sat}_{\texttt{loc}}}
\def\sat{\mathfrak{sat}}

\def\startmodif{\color{blue}}  
\def\stopmodif{\color{black}}

\begin{document}

\maketitle
\slugger{sicon}{xxxx}{xx}{x}{x--x}


%

\renewcommand{\thefootnote}{\fnsymbol{footnote}}

\footnotetext[2]{
Gipsa-lab, Department of Automatic Control, Grenoble Campus, 11 rue des Math\'ematiques, BP 46, 38402 Saint Martin d'H\`eres Cedex, France.
 E-mail: {\tt swann.marx@gipsa-lab.fr, christophe.prieur@gipsa-lab.fr}
 }
\footnotetext[3]{
Departamento de Matem\'atica, Universidad T\'ecnica Federico Santa Mar\'ia, Avda. Espa\~na 1680, Valpara\'iso, Chile. E-mail: {\tt eduardo.cerpa@usm.cl}
}
\footnotetext[4]{
Universit\'e Lyon 1 CNRS UMR 5007 LAGEP, France and Fachbereich C - Mathematik und Naturwissenschaften, Bergische Universit\"at Wuppertal, Gau\ss stra\ss e 20, 42097 Wuppertal, Germany.
 E-mail: {\tt vincent.andrieu@gmail.com}
 }

 \renewcommand{\thefootnote}{\arabic{footnote}}

\begin{abstract}
This article deals with the design of saturated controls in the context of partial differential equations. It focuses on a Korteweg-de Vries equation, which is a nonlinear mathematical model of waves on shallow water surfaces. Two different types of saturated controls are considered. The well-posedness is proven applying a Banach fixed point theorem, using some estimates of this equation and some properties of the saturation function. The proof of the asymptotic stability of the closed-loop system is separated in two cases: i) when the control acts on all the domain, a Lyapunov function together with a sector condition describing the saturating input is used to conclude on the stability; ii) when the control is localized, we argue by contradiction. Some numerical simulations illustrate the stability of the closed-loop nonlinear partial differential equation. 
\end{abstract}


\pagestyle{myheadings}
\thispagestyle{plain}
\markboth{S. MARX, E. CERPA, C. PRIEUR, V. ANDRIEU}{STABILIZATION OF A KORTEWEG-DE VRIES EQUATION WITH SATURATING CONTROL}

\section{Introduction}

In recent decades, a great effort has been made to take into account input saturations in control designs (see e.g \cite{tarbouriech2011book_saturating}, \cite{zacc2003antiwindupLMI} or more recently \cite{laporte2015bounded}). In most applications, actuators are limited due to some physical constraints and the control input has to be bounded. Neglecting the amplitude actuator limitation can be source of undesirable and catastrophic behaviors for the closed-loop system. The standard method to analyze the stability with such nonlinear controls follows a two steps design. First the design is carried out without taking into account the saturation. In a second step, a nonlinear analysis of the closed-loop system is made when adding the saturation. In this way, we often get local stabilization results. Tackling this particular nonlinearity in the case of finite dimensional systems is already a difficult problem. However, nowadays, numerous techniques are available (see e.g. \cite{tarbouriech2011book_saturating,teel1992globalsaturation,sussmann1991saturation}) and such systems can be analyzed with an appropriate Lyapunov function and a sector condition of the saturation map, as introduced in \cite{tarbouriech2011book_saturating}. 

In the \startmodif literature\stopmodif, there are few papers studying this topic in the infinite dimensional case. Among them, we can cite \cite{lasiecka2002saturation}, \cite{prieur2016wavecone}, where a wave equation equipped with a saturated distributed actuator is studied, and \cite{daafouz2014nonlinear}, where a coupled PDE/ODE system modeling a switched power converter with a transmission line is considered. Due to some restrictions on the system, a saturated feedback has to be designed in the latter paper. There exist also some papers using the nonlinear semigroup theory and focusing on abstract systems (\cite{Logemann98time-varyingand},\cite{seidman2001note},\cite{slemrod1989mcss}).

Let us note that in \cite{slemrod1989mcss}, \cite{seidman2001note} and \cite{Logemann98time-varyingand}, the study of a priori bounded controller is tackled using abstract nonlinear theory. To be more specific, for bounded (\cite{slemrod1989mcss},\cite{seidman2001note}) and unbounded (\cite{seidman2001note}) control operators, some conditions are derived to deduce, from the asymptotic stability of an infinite-dimensional linear system in abstract form, the asymptotic stability when closing the loop with saturating controller. These articles use the nonlinear semigroup theory (see e.g. \cite{miyadera1992nl_sg} or \cite{brezis2010functional}). 
 
The Korteweg-de Vries equation (KdV for short)
\begin{equation}
y_t+y_{x}+y_{xxx}+yy_x=0,
\end{equation}
is a mathematical model of waves on shallow water surfaces. Its controllability and stabilizability properties have been deeply studied with no constraints on the control, as reviewed in 
 \cite{cerpa2013control, bible_coron, rosier-zhang}. In this article, we focus on the following controlled  KdV equation
\begin{equation}
\label{nlkdv}
\left\{
\begin{array}{ll}
y_t+y_x+y_{xxx}+yy_x+f=0,& (t,x)\in [0,+\infty)\times [0,L],\\
y(t,0)=y(t,L)=y_x(t,L)=0,& t\in[0,+\infty),\\
y(0,x)=y_0(x),& x\in [0,L],
\end{array}
\right.
\end{equation}
where $y$ stands for the state and $f$ for the control. As studied in \cite{rosier1997kdv}, if $f=0$ and
\begin{equation}
\label{critical-length}
L\in\left\{ 2\pi\sqrt{\frac{k^2+kl+l^2}{3}}\,\Big\slash \,k,l\in\mathbb{N}^*\right\},
\end{equation}
then, there exist solutions of the linearized version of \eqref{nlkdv}, written as follows,
\begin{equation}
\label{lkdv}
\left\{
\begin{array}{l}
y_t+y_x+y_{xxx}=0,\\
y(t,0)=y(t,L)=y_x(t,L)=0,\\
y(0,x)=y_0(x),
\end{array}
\right.
\end{equation}
for which the $L^2(0,L)$-energy does not decay to zero. For instance, if $L=2\pi$ and $y_0=1-\cos(x)$ for all $x\in [0,L]$, then $y(t,x)=1-\cos(x)$ is a stationary solution of \eqref{lkdv} conserving the energy for any time $t$. Note however that, if $L=2\pi$ and $f=0$, the origin of \eqref{nlkdv} is locally asymptotically stable as stated in \cite{chu2013asymptotic}. \startmodif It is worth to mention that there is no hope to obtain global stability, as established in \cite{doronin-natali} where an equilibrium with arbitrary large amplitude is built.\stopmodif

In the literature there are some methods stabilizing the KdV equation \eqref{nlkdv} with boundary  \cite{cerpa2009rapid, cerpa_coron_backstepping, marx-cerpa} or distributed controls \cite{perla-vasconcellos-zuazua, pazoto2005localizeddamping}. Here we focus on the distributed control case. In fact, as proven in \cite{perla-vasconcellos-zuazua, pazoto2005localizeddamping}, the feedback control $f(t,x)=a(x)y(t,x)$, where $a$ is a positive function whose support is a nonempty open subset of $(0,L)$, makes the origin an exponentially stable solution. 

In \cite{mcpa2015kdv_saturating}, in which it is considered a linear Korteweg-de Vries equation with a saturated distributed control, we use a nonlinear semigroup theory. In the case of the present paper, since the term $yy_x$ is not globally Lipschitz, such a theory is harder to use. Thus, we aim here at studying a particular nonlinear partial differential equation without seeing it as an abstract control system and without using the nonlinear semigroup theory. In this paper, we introduce two different types of saturation borrowed from \cite{prieur2016wavecone,mcpa2015kdv_saturating} and \cite{slemrod1989mcss}. In finite dimension, a way to describe this constraint is to use the classical saturation function (see \cite{tarbouriech2011book_saturating} for a good introduction on saturated control problems) defined by
\begin{equation}
\texttt{sat}(s)=\left\{
\begin{array}{rl}
-u_0 &\text{ if } s\leq -u_0,\\
s &\text{ if } -u_0\leq s\leq u_{0},\\
u_0 &\text{ if } s\geq u_{0},
\end{array}
\right.
\end{equation}
for some $u_0>0$.
As in \cite{prieur2016wavecone} and \cite{mcpa2015kdv_saturating} we use its extension to infinite dimension for the following feedback law
\begin{equation}
f(t,x)=\satl(ay)(t,x),
\end{equation}
where, for all sufficiently smooth function $s$ and for all $x\in [0,L]$, $\satl$ is defined as follows
\begin{equation}
\label{sat-linf}
\satl(s)(x)=\texttt{sat}(s(x)).
\end{equation}
Such a saturation is called localized since its image depends only on the value of $s$ at $x$.

\startmodif In this work, we also use a saturation operator in $L^2(0,L)$, denoted by $\sath$, and defined by \stopmodif
\begin{equation}
\label{function-saturation}
\sath(s)(x)=\left\{
\begin{array}{rl}
&\hspace{-1.3cm}s(x)\hspace{0.65cm}\text{ if }\Vert s\Vert_{L^2(0,L)}\leq u_{0},\\
\frac{s(x)u_0}{\Vert s\Vert_{L^2(0,L)}}&\text{ if } \Vert s\Vert_{L^2(0,L)}\geq u_{0}.
\end{array}
\right.
\end{equation}
Note that this definition is borrowed from \cite{slemrod1989mcss} (see also \cite{seidman2001note} or \cite{lasiecka2002saturation}) where the saturation is obtained from the norm of the Hilbert space of the control operator. This saturation seems more natural when studying the stability with respect to an energy, but it is less relevant than $\satl$ for applications. Figure \ref{diff-sat} illustrates how different these saturations are.

\begin{figure}[ht]
\centering
  \includegraphics[scale=0.6]{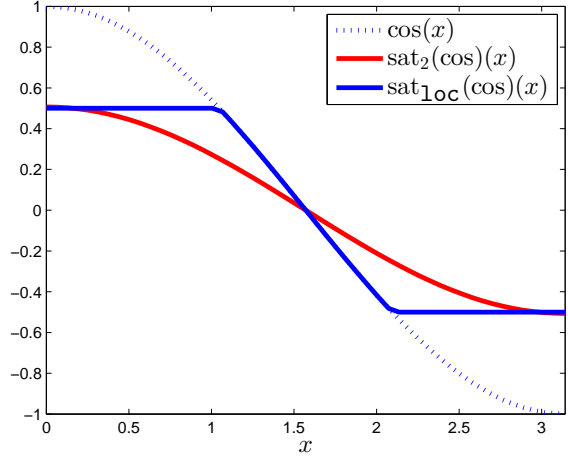}
\caption{$x\in [0,\pi]$. Red: $\sath(\cos)(x)$ and $u_0=0.5$, Blue: $\satl(\cos)(x)$ and $u_0=0.5$, Dotted lines: $\cos(x)$.}
\label{diff-sat}
\end{figure}

Our first main result states that using either the localized saturation \eqref{sat-linf} or using the $L^2$ saturation map \eqref{function-saturation} the KdV equation \eqref{nlkdv} in closed loop with a saturated control is well-posed (see Theorem \ref{nl-theorem-wp} below for a precise statement). Our second main result states that the origin of the KdV equation \eqref{nlkdv} in closed loop with a saturated control is globally asymptotically stable. Moreover, in the case where the control acts on all the domain and where the control is saturated with \eqref{function-saturation}, if the initial conditions are bounded in $L^2$ norm, then the solution converges exponentially with a decay rate that can be estimated (see Theorem \ref{glob_as_stab} below for a precise statement).    

This article is organized as follows. In Section \ref{sec_mainresults}, we present our main results about the well posedness and the stability of \eqref{nlkdv} in presence of saturating control. Sections \ref{sec_wp} and \ref{sec_stab} are devoted to prove these results by using the Banach fixed-point theorem, Lyapunov techniques and a contradiction argument. In Section \ref{sec_simu}, we provide a numerical scheme for the nonlinear equation and give some simulations of the equation looped by a saturated feedback. Section \ref{sec_conc} collects some concluding remarks and possible further research lines. 


\textbf{Notation:} 
\startmodif A function $\alpha$ is said to be a class $\mathcal{K}_\infty$ function if $\alpha$ is nonnegative, increasing, vanishing at $0$ and such that $\lim_{s\rightarrow +\infty}\alpha(s)=+\infty$.\stopmodif

\section{Main results}

\label{sec_mainresults}

We first give an analysis of our system \eqref{nlkdv} when there is no constraint on the control $f$. To do that, letting $f(t,x):=ay(t,x)$ in \eqref{nlkdv}, where $a$ is a nonnegative function satisfying
\begin{equation}
\label{gain-control}
\left\{
\begin{array}{l}
0<a_0 \leq a(x)\leq a_1,\quad \forall x\in\omega,\\
\text{where $\omega$ is a nonempty open subset of $(0,L)$},
\end{array}
\right.
\end{equation}
then, following \cite{rosier2006global}, we get that the origin of \eqref{nlkdv} is globally asymptotically stabilized. If $\omega=[0,L]$, then any solution to \eqref{nlkdv} satisfies
\begin{equation}
\label{lyap-without-sat}
\frac{1}{2}\frac{d}{dt}\int_0^L |y(t,x)|^2 dx=-\frac{1}{2}|y_x(t,0)|^2-\int_0^L a(x) |y(t,x)|^2 dx
\leq  -a_0\int_0^L |y(t,x)|^2 dx,
\end{equation} 
which ensures an exponential stability with respect to the $L^2(0,L)$-norm. Note that the decay rate can be selected as large as we want by tuning the parameter $a_0$. Such a result is refered to as a rapid stabilization result.

Let us consider the KdV equation controlled by a saturated distributed control as follows
\begin{equation}
\label{nlkdv_sat}
\left\{
\begin{array}{l}
y_t+y_x+y_{xxx}+yy_x+\sat(ay)=0,\\
y(t,0)=y(t,L)=y_x(t,L)=0,\\
y(0,x)=y_0(x),
\end{array}
\right.
\end{equation}
where $\sat=\sath$ or $\satl$. Since these two operators have properties in common, we will use the notation $\sat$ all along the paper. However, in some cases, we get different results. Therefore, the use of a particular saturation is specified when it is necessary. 

Let us state the main results of this paper.

\begin{theorem}\em (Well posedness) \em
\label{nl-theorem-wp}
For any initial condition $y_0\in L^2(0,L)$, there exists a unique mild solution $y\in C([0,T];L^2(0,L))\cap L^2(0,T;H^1(0,L))$ to \eqref{nlkdv_sat}. 
\end{theorem}

\begin{theorem} \em (Global asymptotic stability) \em
\label{glob_as_stab}
Given a nonempty open subset $\omega$ and the positive values $a_0$ and $u_0$, there exist a positive value $\mu^{\star}$ and a \startmodif class $\mathcal{K}_\infty$ function $\alpha_0:\mathbb{R}_{\geq 0}\rightarrow \mathbb{R}_{\geq 0}$ \stopmodif such that for any $y_0\in L^2(0,L)$, the mild solution $y$ of (\ref{nlkdv_sat}) satisfies
\begin{equation}
\Vert y(t,.)\Vert_{L^2(0,L)}\leq \alpha_0(\Vert y_0\Vert_{L^2(0,L)})e^{-\mu^{\star}t},\quad \forall t\geq 0.
\end{equation}

Moreover, in the case where $\omega=[0,L]$ and $\sat=\sath$ we can estimate \em locally \em the decay rate of the solution. In other words, for all $r>0$, for any initial condition $y_0\in L^2(0,L)$ such that $\Vert y_0\Vert_{L^2(0,L)}\leq r$, the mild solution $y$ to \eqref{nlkdv_sat} satisfies
\begin{equation}
\label{local_as_stab_estim}
\Vert y(t,.)\Vert_{L^2(0,L)}\leq \Vert y_0\Vert_{L^2(0,L)}e^{-\mu t},\quad \forall t\geq 0,
\end{equation}
where $\mu$ is defined as follows
\begin{equation}
\label{formula-mu}
\mu:=\min\left\{a_0,\frac{u_0a_0}{ra_1}\right\}.
\end{equation} 
\end{theorem}

The remaining part of this paper is devoted to the proof of these results (see Sections \ref{sec_wp} and \ref{sec_stab}, respectively) and to numerical simulations to illustrate Theorem \ref{glob_as_stab} (see Section \ref{sec_simu}). 


\section{Well-posedness}

\label{sec_wp}

\subsection{Linear system}

\label{linear_wp}

Before proving the well-posedness of \eqref{nlkdv_sat}, let us recall some useful results on the linear system \eqref{lkdv}. To do that, consider the operator defined by
$$
D(A)=\lbrace w\in H^3(0,L),\: w(0)=w(L)=w^\prime(L)=0\rbrace,
$$
$$
A:w\in D(A)\subset L^2(0,L)\longmapsto (-w^\prime-w^{\prime\prime\prime})\in L^2(0,L).
$$
It can be proved that this operator and its adjoint operator defined by
$$
D(A^\star)=\lbrace w\in H^3(0,L),\: w(0)=w(L)=w^\prime(0)=0\rbrace,
$$
$$
A^\star:w\in D(A^\star)\subset L^2(0,L)\longmapsto w^\prime+w^{\prime\prime\prime},
$$
are both dissipative, which means that, for all $w\in D(A)$, $\int_0^L wA(w)dx\leq 0$ and for all $w\in D(A^\star)$, $\int_0^L wA^\star(w)dx\leq 0$.

Therefore, from \cite{pazy1983semigroups}, the operator $A$ generates a strongly continuous semigroup of contractions which we denote by $W(t)$. We have the following theorem proven in \cite{rosier1997kdv} and \cite{cerpa2013control}
\begin{theorem}[Well-posedness of \eqref{lkdv}, \cite{rosier1997kdv},\cite{cerpa2013control}]
\label{lkdv-wp}
\begin{itemize}
\item For any initial condition $y_0\in D(A)$, there exists a unique strong solution $y\in C(0,T;D(A))\cap C^1(0,T;L^2(0,L))$ to (\ref{lkdv}).
\item For any initial condition $y_0\in L^2(0,L)$, there exists a unique mild solution $y\in C([0,T];L^2(0,L))\cap L^2(0,T;H^1(0,L))$ to (\ref{lkdv}). Moreover, there exists $C_0>0$ such that the solution to (\ref{lkdv}) satisfies
\begin{equation}
\label{dissipativity-regularity}
\Vert y\Vert_{C(0,T;L^2(0,L))}+\Vert y\Vert_{L^2(0,T;H^1(0,L))}\leq C_0\Vert y_0\Vert_{L^2(0,L)}
\end{equation}
and the extra trace regularity
\begin{equation}
\label{extra-regularity}
\Vert y_x(.,0)\Vert_{L^2(0,T)}\leq \Vert y_0\Vert_{L^2(0,L)}.
\end{equation}
\end{itemize}
\end{theorem}

To ease the reading, let us denote the following Banach space, for all $T>0$, 
$$\mathcal{B}(T):=C(0,T;L^2(0,L))\cap L^2(0,T;H^1(0,L))$$
endowed with the norm 
\begin{equation}
\Vert y\Vert_{\mathcal{B}(T)}=\sup_{t\in [0,T]}\Vert y(t,.)\Vert_{L^2(0,L)}+\left(\int_0^T \Vert y(t,.)\Vert^2_{H^1(0,L)}dt\right)^{\frac{1}{2}}.
\end{equation}

Before studying the well-posedness of (\ref{nlkdv_sat}), we need a well-posedness result with a right-hand side. Given $g\in L^1(0,T;L^2(0,L)$, let us consider $y$ the unique solution \footnote{With $f=0$, the existence and the unicity of $y$ are insured since $A$ generates a $C_0$-semigroup of contractions. It follows from the semigroup theory the existence and the unicity of $y$ when $g\in L^1(0,T;L^2(0,L))$ (see \cite{pazy1983semigroups}).} to the following nonhomogeneous problem:
\begin{equation}
\label{kdv-2}
\left\{
\begin{array}{l}
y_t+y_x+y_{xxx}=g,\\
y(t,0)=y(t,L)=y_x(t,L)=0,\\
y(0,.)=y_0.
\end{array}
\right.
\end{equation}

Note that we need the following property on the saturation function, which will allow us to state that this type of nonlinearity belongs to the space $L^1(0,T;L^2(0,L))$.
\begin{lemma}
\label{lipschitz-satl2}
For all $(s,\tilde{s})\in L^2(0,L)^2$, we have
\begin{equation}
\Vert \sat(s)-\sat(\tilde{s})\Vert_{L^2(0,L)}\leq 3\Vert s-\tilde{s}\Vert_{L^2(0,L)}.
\end{equation}
\end{lemma}

\begin{proof}
For $\sat=\sath$, please refer to \cite[Theorem 5.1.]{slemrod1989mcss} for a proof. For $\sat=\satl$, we know from \cite[Page 73]{bible_khalil} that for all $(s,\tilde{s})\in L^2(0,L)^2$ and for all $x\in [0,L]$, 
$$|\satl(s(x))-\satl(\tilde{s}(x))|\leq |s(x)-\tilde{s}(x)|.$$
Thus, we get 
$$\Vert \satl(s)-\satl(\tilde{s})\Vert_{L^2(0,L)}\leq \Vert s-\tilde{s}\Vert_{L^2(0,L)},$$
which concludes the proof of Lemma \ref{lipschitz-satl2}. 
\end{proof}

\startmodif We have the following proposition borrowed from \cite[Proposition 4.1]{rosier1997kdv}. \stopmodif

\begin{proposition}[\cite{rosier1997kdv}]
\label{proposition-reg-rosier}
If $y\in L^2(0,T;H^1(0,L))$, then $yy_x\in L^1(0,T;L^2(0,L))$ and the map $\psi_1: y\in L^2(0,T;H^1(0,L))\mapsto yy_x\in L^1(0,T;L^2(0,L))$ is continuous
\end{proposition}

We have also the following proposition.
\begin{proposition}
\label{proposition-reg}
\startmodif Assume $a:[0,L]\rightarrow \mathbb{R}$ satisfies \eqref{gain-control}\stopmodif. If $y\in L^2(0,T;H^1(0,L))$, then  $\sat(ay)\in L^1(0,T;L^2(0,L))$ and the map $\psi_2: y\in L^2(0,T;H^1(0,L)) \mapsto \sat(ay)\in L^1(0,T;L^2(0,L))$ is continuous;
\end{proposition}
\begin{proof}
Let $y,z\in L^2(0,T;H^1(0,L))$. We have, using Lemma \ref{lipschitz-satl2} and H\"older inequality
\startmodif
\begin{eqnarray}
\label{regularity-sat-l1}
\Vert \sat(ay) - \sat(az)\Vert_{L^1(0,T;L^2(0,L))} &&\leq 3\int_0^T \Vert a(y-z)\Vert_{L^2(0,L)}\nonumber \\
&&\leq 3\sqrt{L}a_1\sqrt{T}\Vert (y-z)\Vert_{L^2(0,T;H^1(0,L))}
\end{eqnarray}
\stopmodif
Plugging $z=0$ in (\ref{regularity-sat-l1}) yields $\sat(ay)\in L^1(0,T;L^2(0,L))$ and (\ref{regularity-sat-l1}) implies the continuity of the map $\psi_2$. It concludes the proof of Proposition \ref{proposition-reg}.\end{proof}
\startmodif Let us study the non-homogenenous linear KdV equation with $y_0(x):=0$. For any $g\in L^1(0,T;L^2(0,L))$, it is described with the following equation
\begin{equation}
\label{kdv-zero}
\left\{
\begin{split}
&y_t+y_x+y_{xxx}+g=0,\\
&y(t,0)=y(t,L)=y_x(t,L)=0,\\
&y(0,x)=0.
\end{split}
\right.
\end{equation}
It can be rewritten as follows
\begin{equation}
\left\{
\begin{split}
&\dot{y}=Ay+g,\\
&y(0)=0.
\end{split}
\right.
\end{equation}
By standard semigroup theory (see \cite{pazy1983semigroups}), for any positive value $t$ and any function $g\in L^1(\mathbb{R}_{\geq 0};L^2(0,L))$, the solution to \eqref{kdv-zero} can be expressed as follows
\begin{equation}
y(t)=\int_0^t W(t-\tau)g(\tau,x)d\tau.
\end{equation}
Finally, we have the following result borrowed from \cite[Lemma 2.2]{rosier2006global}
\begin{proposition}[\cite{rosier2006global}]
\label{prop-w(t-tau)}
There exists a positive value $C_1$ such that for any positive value $T$ and any function $g\in L^1(0,T;L^2(0,L))$ the solution to \eqref{kdv-zero} satisfies the following inequality, 
\begin{equation}
\left\Vert \int_0^t W(t-\tau)g(\tau,x)d\tau\right\Vert_{\mathcal{B}(T)} \leq C_1\int_0^T \Vert g(\tau,.)\Vert_{L^2(0,L)}d\tau.
\end{equation}
\end{proposition}
\stopmodif

\subsection{Proof of Theorem \ref{nl-theorem-wp}}

Let us begin this section with a technical lemma.
\begin{lemma}\em (\cite{zhang1999KdV}) \em
\label{zhang-regularity}
For any $T>0$ and $y,z\in\mathcal{B}(T)$,
\begin{equation}
\int_0^T \Vert (y(t,.)z(t,.))_x\Vert_{L^2(0,L)}dt\leq 2 \sqrt{T}\Vert y\Vert_{\mathcal{B}(T)}\Vert z\Vert_{\mathcal{B}(T)}.
\end{equation}
\end{lemma}

The following is a local well-posedness result.
\begin{lemma}\em (Local well-posedness) \em
\label{local-wp}
Let $T>0$ be given. For any $y_0\in L^2(0,L)$, there exists $T^\prime \in [0,T]$ depending on $\Vert y_0\Vert_{L^2(0,L)}$ such that \eqref{nlkdv_sat} admits a unique mild solution $y\in\mathcal{B}(T^\prime)$. 
\end{lemma}

\begin{proof}

We follow the strategy of \cite{chapouly2009global} and \cite{rosier2006global}. We know from Proposition \ref{proposition-reg} that, for all $z\in L^1(0,T;L^2(0,L))$, there exists a unique mild solution to the following system
\begin{equation}
\label{kdv-fixed-point}
\left\{
\begin{array}{l}
y_t+y_x+y_{xxx}=-zz_x-\sat(az),\\
y(t,0)=y(t,L)=y_x(t,L)=0,\\
y(0,x)=y_0(x).
\end{array}
\right.
\end{equation}
Solution to \eqref{kdv-fixed-point} can be written in its integral form
\begin{equation}
y(t)=W(t)y_0-\int_0^t W(t-\tau)(zz_x)(\tau)d\tau-\int_0^t W(t-\tau)\sat(az(\tau,.))d\tau.
\end{equation}
For given $y_0\in L^2(0,L)$, let $r$ and $T^\prime$ be positive constants to be chosen later. We define
\begin{equation}
S_{T^\prime ,r}=\lbrace z\in \mathcal{B}(T^\prime),\: \Vert z\Vert_{\mathcal{B}(T^\prime)}\leq r\rbrace,
\end{equation}
which is a closed, convex and bounded subset of $\mathcal{B}(T^\prime)$. Consequently, $S_{T^\prime ,r}$ is a complete metric space in the topology induced from $\mathcal{B}(T^\prime)$. We define a map $\Gamma$ on $S_{T^\prime,r}$ by, for all $t\in [0,T^\prime]$
\begin{equation}
\Gamma(z):=W(t)y_0-\int_0^t W(t-\tau)(zz_x)(\tau)d\tau-\int_0^t W(t-\tau)\sat(az(\tau,.))d\tau,\: \forall z\in S_{T^\prime,r}.
\end{equation}
We aim at proving that there exists a unique fixed point to this operator. \startmodif It follows from Proposition \ref{prop-w(t-tau)}, Lemma \ref{zhang-regularity} and the linear estimates given in Theorem \ref{lkdv-wp} that for every $z\in S_{T^\prime,r}$, there exists a positive value $C_2:=C_2(a_1,T,L,C_1)$ such that it holds
\begin{equation}
\begin{split}
\Vert \Gamma(z)\Vert_{\mathcal{B}(T^\prime)} &\leq C_0 \Vert y_0\Vert_{L^2(0,L)}+C_1\int_0^T (\Vert zz_x(\tau,.)\Vert_{L^2(0,L)}+\Vert \sat(az(\tau,.)\Vert_{L^2(0,L)}) d\tau\\ 
& \leq C_0\Vert y_0\Vert_{L^2(0,L)} +2C_1\sqrt{T^\prime} \Vert z\Vert^2_{\mathcal{B}(T^\prime)}+C_2\sqrt{T^\prime}\Vert z \Vert_{\mathcal{B}(T^\prime)}
\end{split}
\end{equation}
where the first line has been obtained with the linear estimates given in Theorem \ref{lkdv-wp} and the estimate given in Proposition \ref{prop-w(t-tau)} and the second line with Lemma \ref{zhang-regularity} and Proposition \ref{proposition-reg}.
We choose $r>0$ and $T^\prime>0$ such that 
\begin{equation}
\left\{
\begin{array}{l}
r=2C_0\Vert y_0\Vert_{L^2(0,L)},\\
2C_1\sqrt{T^\prime}r + C_2\sqrt{T^\prime}\leq \frac{1}{2},
\end{array}
\right.
\end{equation}
\stopmodif
in order to obtain
\begin{equation}
\Vert \Gamma(z)\Vert_{\mathcal{B}(T^\prime)}\leq r,\quad \forall z\in S_{T^\prime,r}.
\end{equation}
Thus, with such $r$ and $T^\prime$, $\Gamma$ maps $S_{T^\prime ,r}$ to $S_{T^\prime ,r}$. Moreover, one can prove \startmodif with Proposition \ref{prop-w(t-tau)}, Lemma \ref{zhang-regularity} and the linear estimates given in Theorem \ref{lkdv-wp}\stopmodif that
\begin{equation}
\Vert \Gamma(z_1)-\Gamma(z_2)\Vert_{\mathcal{B}(T^\prime)}\leq \frac{1}{2}\Vert z_1-z_2\Vert_{\mathcal{B}(T^\prime)},\: \forall z_1,z_2\in S_{T^\prime ,r}.
\end{equation}
The existence of mild solutions to the Cauchy problem \eqref{nlkdv_sat} follows by using the Banach fixed-point theorem \cite[Theorem 5.7]{brezis2010functional}. \end{proof}
Before proving the global well-posedness, we need the following lemma inspired by \cite{coroncrepeau2004missed} and \cite{chapouly2009global} which implies that if there exists a solution for some $T>0$ then the solution is unique.
\begin{lemma}
\label{uniqueness-coron-crepeau}
\startmodif Let $T>0$ and $a:[0,L]\rightarrow \mathbb{R}$ satisfying \eqref{gain-control}. There exists $C_{11}:=C_{11}(T,L)>0$ such that for every $y_0,z_0\in L^2(0,L)$ for which there exist mild solutions $y$ and $z$ of
\begin{equation}
\label{kdv-y}
\left\{
\begin{array}{l}
y_t+y_x+y_{xxx}+yy_x+\sat(ay)=0,\\
y(t,0)=y(t,L)=y_x(t,L)=0,\\
y(0,x)=y_0(x),
\end{array}
\right.
\end{equation}
and
\begin{equation}
\label{kdv-z}
\left\{
\begin{array}{l}
z_t+z_x+z_{xxx}+zz_x+\sat(az)=0,\\
z(t,0)=z(t,L)=z_x(t,L)=0,\\
z(0,x)=z_0(x),
\end{array}
\right.
\end{equation}
these solutions satisfy
\begin{equation}
\int_0^T \int_0^L (z_x(t,x)-y_x(t,x))^2dxdt\leq e^{C_{11}(1+\Vert y\Vert_{L^2(0,T;H^1(0,L))}+\Vert z\Vert_{L^2(0,T;H^1(0,L))})}\int_0^L (z_0(x)-y_0(x))^2dx,
\end{equation} 
\begin{equation}
\int_0^T \int_0^L (z(t,x)-y(t,x))^2dxdt\leq e^{C_{11}(1+\Vert y\Vert_{L^2(0,T;H^1(0,L))}+\Vert z\Vert_{L^2(0,T;H^1(0,L))})}\int_0^L (z_0(x)-y_0(x))^2dx.
\end{equation}
\stopmodif
\end{lemma} 

\begin{proof}

We follow the strategy of \cite{coroncrepeau2004missed} and \cite{chapouly2009global}. Let us assume that for given $y_0\in L^2(0,L)$, there exist $T>0$ and two different solutions $y$ and $z$ to \eqref{kdv-y} and \eqref{kdv-z}, respectively, defined on $[0,T]\times [0,L]$. Then $\Delta:=z-y$ defined on $[0,T]\times [0,L]$ is a mild solution of
\begin{equation}
\label{kdv-delta}
\left\{
\begin{array}{l}
\Delta_t+\Delta_x+\Delta_{xxx}=-y\Delta_x-z_x\Delta-(\sat(az)-\sat(ay)),\\
\Delta(t,0)=\Delta(t,L)=\Delta_x(t,L)=0,\\
\Delta(0,x)=z_0(x)-y_0(x).
\end{array}
\right.
\end{equation}
Integrating by parts in
\begin{equation}
\int_0^L 2x\Delta(\Delta_t+\Delta_x+\Delta_{xxx}+y\Delta_x+z_x\Delta+\sat(az)-\sat(ay))dx=0,
\end{equation}
and using the boundary conditions of \eqref{kdv-delta}, we readily get
\begin{multline}
\frac{d}{dt}\int_0^L x\Delta^2dx+3\int_0^L \Delta^2_xdx= \int_0^L \Delta^2 dx-2\int_0^L xy\Delta\Delta_xdx\\
+\int_0^L z\Delta^2 dx+4\int_0^L xz\Delta\Delta_xdx-\int_0^L x\Delta(\sat(az)-\sat(ay))dx.
\end{multline}
By the boundary conditions and the continuous Sobolev embedding $H^1_0(0,L)\subset C([0,T])$, there exists $C_3=C_3(L)>0$ such that
\begin{equation}
2\left|\int_0^L xy\Delta\Delta_xdx\right|\leq C_3\Vert y_x\Vert_{L^2(0,L)}\int_0^L |x\Delta\Delta_x|dx.
\end{equation}
Thus,
\begin{equation}
2\left|\int_0^L xy\Delta\Delta_xdx\right|\leq \frac{1}{2}\int_0^L \Delta_x^2dx+\frac{C_3^2}{2}\Vert y_x\Vert^2_{L^2(0,L)}L\int_0^L x\Delta^2dx.
\end{equation}
Similarly,
\begin{equation}
4\left|\int_0^L xz\Delta\Delta_xdx\right|\leq \frac{1}{2}\int_0^L \Delta_x^2 dx+2C_3^2\Vert z_x\Vert^2_{L^2(0,L)}\int_0^L x\Delta^2dx.
\end{equation}
Moreover, since $\sat$ is globally \startmodif Lipschitz with constant 3 \stopmodif (as stated in Lemma \ref{lipschitz-satl2}) and \startmodif for all $x\in [0,L]$, $a(x)\leq a_1$, we use a H\"older inequality to get \stopmodif
\begin{equation}
\begin{array}{rcl}
\left|\int_0^L x\Delta(\sat(az)-\sat(ay))dx\right| &\leq& \Vert x\Delta\Vert_{L^2(0,L)}\Vert \sat(az)-\sat(ay)\Vert_{L^2(0,L)}\\
&\leq& 3\Vert a(x)\Delta\Vert_{L^2(0,L)}\Vert x\Delta\Vert_{L^2(0,L)}\\
&\leq& 3a_1\int_0^L x\Delta^2 dx.
\end{array}
\end{equation}
Note that, from \cite[Lemma 16]{coroncrepeau2004missed}, for every $\phi\in H^1(0,L)$ with $\phi(0)=0$, and every $d\in [0,L]$,
\begin{equation}
\label{lemma16-cc}
\int_0^L \phi^2dx\leq \frac{d^2}{2}\int_0^L \phi_x^2dx+\frac{1}{d}\int_0^L x\phi^2dx.
\end{equation}
Thus, from \eqref{lemma16-cc} there exists $C_4>0$ such that
$$
\int_0^L \Delta^2dx\leq \frac{1}{2}\int_0^L \Delta_x^2dx+C_4\int_0^L x\Delta^2dx.
$$
Moreover, with the boundary conditions of $z$ and the Sobolev embedding $H_0^1(0,L)\subset C([0,T])$, there exists $C_5=C_5(L)>0$ such that
$$
2\left|\int_0^L z\Delta^2dx\right|\leq C_5\Vert z_x\Vert_{L^2(0,L)}\int_0^L \Delta^2dx.
$$
Hence, using the boundary conditions of $\Delta$ and \eqref{lemma16-cc} with $d:=\min\{C_5^{-1/2}\Vert z_x\Vert^{-1/2}_{L^2(0,L)},L\}$, there exists $C_6=C_6(L)>0$ such that
\begin{equation}
2\int_0^L z\Delta^2 dx\leq \frac{1}{2}\int_0^L\Delta_x^2dx+C_6(1+\Vert z_x\Vert_{L^2(0,L)}^{3/2})\int_0^L x\Delta^2dx.
\end{equation}
Finally, there exists $C_7=C_7(L)>0$ such that
\begin{equation}
\frac{d}{dt}\int_0^L x\Delta^2dx+\int_0^L \Delta_x^2dx\leq C_7(1+\Vert y_x\Vert^2_{L^2(0,L)}+\Vert z_x\Vert^2_{L^2(0,L)})\int_0^L x\Delta^2dx.
\end{equation}
In particular,
\begin{equation}
\frac{d}{dt}\int_0^L x\Delta^2dx\leq C_7(1+\Vert y_x\Vert^2_{L^2(0,L)}+\Vert z_x\Vert^2_{L^2(0,L)})\int_0^L x\Delta^2dx.
\end{equation}
Using the Gr\"onwal Lemma, the last inequality and the initial conditions of $\Delta$, we get, for every $t\in[0,T]$,
\begin{equation}
\int_0^L x\Delta^2(t,x)dx\leq e^{C_7\left( T+\Vert y\Vert^2_{L^2(0,T;H^1(0,L))}+\Vert z\Vert^2_{L^2(0,T;H^1(0,L))}\right)}\int_0^L x(z_0(x)-y_0(x))^2dx,
\end{equation}
and thus, we obtain the existence of $C_8=C_8(T,L)$ such that
\begin{equation}
\label{unicity-1}
\int_0^T\int_0^L (z_x(t,x)-y_x(t,x))^2dxdt\leq e^{C_8\left(\Vert y\Vert^2_{L^2(0,T;H^1(0,L))}+\Vert z\Vert^2_{L^2(0,T;H^1(0,L))}\right)} \int_0^L (z_0(x)-y_0(x))^2dx.
\end{equation}
Similarly, integrating by parts in
\begin{equation}
\int_0^L \Delta(\Delta_t+\Delta_x+\Delta_{xxx}+y\Delta_x+z_x\Delta+\sat(az)-\sat(ay))dx=0
\end{equation}
we get, using the boundary conditions of $\Delta$,
\begin{equation}
\frac{1}{2}\frac{d}{dt}\int_0^L \Delta^2 dx+\frac{1}{2}\Delta^2_x(t,0)=-\int_0^L (y\Delta_x-2z\Delta_x)\Delta dx-\int_0^L \Delta(\sat(az)-\sat(ay))dx.
\end{equation}
Moreover,
\begin{equation}
\label{unicity-linf}
-\int_0^L (y\Delta_x-2z\Delta_x)\Delta\leq \int_0^L \Delta_x^2dx+\int_0^L \left(\frac{1}{2}y^2+2z^2\right)\Delta^2dx,
\end{equation}
and
\begin{equation}
\label{lipschitz-sat}
\left|\int_0^L \Delta(\sat(az)-\sat(ay))dx\right|\leq 3a_1\int_0^L \Delta^2 dx.
\end{equation}
Thanks to the continuous Sobolev embedding $H^1_0(0,L)\subset C([0,L])$, \eqref{lipschitz-sat} and \eqref{unicity-linf}, there exists $C_9=C_9(L)>0$ such that
\begin{equation}
\frac{1}{2}\frac{d}{dt}\int_0^L \Delta^2 dx\leq \int_0^L \Delta_x^2dx+C_9\left(\Vert y_x\Vert^2_{L^2(0,L)}+\Vert z_x\Vert^2_{L^2(0,L)}+1\right)\int_0^L \Delta^2 dx.
\end{equation}
Thus applying the Gr\"onwall Lemma, we get the existence of $C_{10}=C_{10}(L)>0$ such that
\begin{equation}
\label{unicity-2}
\int_0^L (z(t,x)-y(t,x))^2dx\leq e^{C_{10}\left(1+\Vert y\Vert^2_{L^2(0,T;H^1(0,L))}+\Vert z\Vert^2_{L^2(0,T;H^1(0,L))}\right)}\int_0^L (z_0(x)-y_0(x))^2dx.
\end{equation}
With the use \eqref{unicity-1} and \eqref{unicity-2}, it concludes the proof of Lemma \ref{uniqueness-coron-crepeau}. 
\end{proof}

We aim at removing the smallness condition given by $T^\prime$ in Lemma \ref{local-wp}, following \cite{chapouly2009global}. Since we have the local well-posedness, we only need to prove the following a priori estimate for any mild solution to (\ref{nlkdv_sat}). 

\begin{lemma}
\label{global-estimation}
For given $T>0$, there exists $G:=G(T)>0$ such that for any $y_0\in L^2(0,L)$, for any $0< T^\prime\leq T$ and for any mild solution $y\in \mathcal{B}(T^\prime)$ to (\ref{nlkdv_sat}), it holds
\begin{equation}
\label{glob-estim-chapouly}
\Vert y\Vert_{\mathcal{B}(T^\prime)}\leq G\Vert y_0\Vert_{L^2(0,L)},
\end{equation}
and
\begin{equation}
\label{l2-dissipativity}
\Vert y\Vert_{L^2(0,L)}\leq \Vert y_0\Vert_{L^2(0,L)}.
\end{equation}
\end{lemma}

\begin{proof}
Let us fix $0<T^\prime\leq T$. We multiply the first equation in (\ref{nlkdv_sat}) by $y$ and integrate on $(0,L)$. Using the boundary conditions in \eqref{nlkdv_sat}, we get the following estimates
\begin{equation*}
\int_0^L yy_xdx=0,\hspace{0.3cm}\int_0^L yy_{xxx}dx=\frac{1}{2}|y_x(t,0)|^2,\hspace{0.3cm}\int_0^L y^2y_xdx=0.
\end{equation*}
Using the fact that $\sat$ is odd, we get that
\begin{equation}
\label{l2-dissipativity1}
\frac{1}{2}\frac{d}{dt}\Vert y(t,.)\Vert^2_{L^2(0,L)} \leq -\frac{1}{2}|y_x(t,0)|^2-\int_0^L y\sat(ay)dx\leq 0
\end{equation}
which implies \eqref{l2-dissipativity}. Moreover, using again \eqref{l2-dissipativity1}, there exists $C_{12}=C_{12}(L)>0$ such that
\begin{equation}
\label{global-wp-1}
\Vert y\Vert_{L^\infty(0,T^\prime;L^2(0,L)}\leq C_{12}\Vert y_0\Vert_{L^2(0,L)}.
\end{equation}
It remains to prove a similar inequality for $\Vert y_x\Vert_{L^2(0,T^\prime;L^2(0,L))}$ to achieve the proof. We multiply (\ref{nlkdv_sat}) by $xy$, integrate on $(0,L)$ and use the following
\begin{equation*}
\int_0^L xyy_xdx=-\frac{1}{2}\Vert y\Vert^2_{L^2(0,L)},
\quad 
\int_0^L yy_{xxx}dx=\frac{3}{2}\Vert y_x\Vert^2_{L^2(0,L)},
\end{equation*} and
\begin{multline}
\label{nl-kdv-diss}
-\int_0^L xy^2y_xdx =\frac{1}{3}\int_0^L y^3(t,x)dx
\leq \frac{1}{3}\sup_{x\in [0,L]}|y(t,x)|\Vert y\Vert^2_{L^\infty(0,T;L^2(0,L))}\\
\leq \frac{\sqrt{L}}{3}\Vert y_x\Vert_{L^2(0,L)}\Vert y\Vert^2_{L^\infty(0,T;L^2(0,L))}
\leq \frac{\sqrt{L}\delta}{6}\Vert y_x\Vert_{L^2(0,L)}+\frac{\sqrt{L}}{6\delta}\Vert y\Vert^4_{L^\infty(0,T;L^2(0,L))}
\end{multline}
where $\delta$ is chosen as $\delta:=\frac{3}{\sqrt{L}}$. In this way, we obtain
\begin{equation}
\frac{1}{2}\frac{d}{dt}\int_0^L |x^{1/2}y(t,.)|^2dx-\frac{1}{2} \int_0^L y^2dx+\frac{3}{2}\int_0^L |y_x|^2dx-\frac{1}{3}\int_0^L |y|^3dx =-\int_0^L x\sat(ay)ydx.
\end{equation}
We get, using \eqref{nl-kdv-diss} and the fact that $\sat$ is odd, that
\begin{equation}
\label{H1-inequality}
\frac{1}{2}\frac{d}{dt}\int_0^L |x^{1/2}y(t,.)|^2dx +\int_0^L |y_x|^2dx \leq \frac{1}{2}\Vert y \Vert^2_{L^\infty(0,T;L^2(0,L))}+\frac{L}{18}\Vert y\Vert^4_{L^\infty(0,T;L^2(0,L))}.
\end{equation}
Using (\ref{global-wp-1}) and \startmodif Gr\"onwall inequality\stopmodif, we get the existence of a positive value $C_{13}=C_{13}(L)>0$ such that
\begin{equation}
\Vert y_x\Vert_{L^2(0,T^\prime;L^2(0,L))}\leq C_{13}\Vert y_0\Vert_{L^2(0,L)},
\end{equation}
which concludes the proof of Lemma \ref{global-estimation}. 
\end{proof}
Using a classical extension argument, Lemmas \ref{local-wp}, \ref{global-estimation} and \ref{uniqueness-coron-crepeau}, for any $T>0$, we can conclude that there exists a unique mild solution in $\mathcal{B}(T)$ to (\ref{nlkdv_sat}). Indeed, with Lemma \ref{local-wp}, we know that there exists $T^\prime\in (0,T)$ such that there exists a unique solution to \eqref{nlkdv_sat} in $\mathcal{B}(T^\prime)$. Moreover, Lemma \ref{global-estimation} allows us to state the existence of \startmodif a mild solution \stopmodif to \eqref{nlkdv_sat} for every $T>0$: since the solution $y$ to \eqref{nlkdv_sat} is bounded by its initial condition for every $T^\prime>0$ belonging to $[0,T]$ as stated in \eqref{l2-dissipativity}, we know that there exists a solution to \eqref{nlkdv_sat} in $\mathcal{B}(T)$. Finally, Lemma \ref{uniqueness-coron-crepeau} implies that there exists a unique mild solution to \eqref{nlkdv_sat} in $\mathcal{B}(T)$. This concludes the proof of Theorem \ref{nl-theorem-wp}.

\startmodif 
\begin{remark}
\label{remark-gkdv}
In \cite{rosier2006global}, the following generalized Korteweg-de Vries equation is considered
\begin{equation}
\label{gkdv-rosier}
\left\{
\begin{split}
&y_t+y_x+y_{xxx}+b(y)y_x+ay=0,\\
&y(t,0)=y(t,L)=y_x(t,L)=0,\\
&y(0,x)=y_0(x),
\end{split}
\right.
\end{equation}
where the function $a:[0,L]\rightarrow \mathbb{R}$ satisfies \eqref{gain-control} and where $b:\mathbb{R}\rightarrow \mathbb{R}$ satisfies the following growth condition
\begin{equation}
b(0)=0,\: |b^{(j)}(\mu)|\leq C\left(1+|\mu|^{p-j}\right),\quad \forall \mu\in\mathbb{R},
\end{equation}
for $j=0$ if $1\leq p< 2$ and for $j=0,1,2$ if $p\geq 2$. 

The saturated version of \eqref{gkdv-rosier} is
\begin{equation}
\label{gkdv-saturated}
\left\{
\begin{split}
&y_t+y_x+y_{xxx}+b(y)y_x+\sat(ay)=0,\\
&y(t,0)=y(t,L)=y_x(t,L)=0,\\
&y(0,x)=y_0(x).
\end{split}
\right.
\end{equation}
The strategy followed in \cite{rosier2006global} can be followed easily to prove the same result than Theorem \ref{nl-theorem-wp} for \eqref{gkdv-saturated}. Note that in \cite{rosier2006global}, provided that the initial condition satisfies some compatibility conditions, the well-posedness is proved for a solutions in $C([0,T];H^s(0,L))\cap L^2(0,T;H^{s+1}(0,L))$, where $s\in [0,3]$. The authors proved this result by looking at $v=y_t$ which solves an equation equivalent to \eqref{gkdv-rosier}. In our case, it seems harder to prove such a result. Since the saturation operator introduces some non-smoothness, $v=y_t$ does not solve an equation equivalent to \eqref{gkdv-saturated}. 
\end{remark}
\stopmodif
\section{Global asymptotic stability}

\label{sec_stab}

Let us begin by introducing the following definition.
\begin{definition}
System (\ref{nlkdv_sat}) is said to be \em semi-globally exponentially stable \em in $L^2(0,L)$ if for any $r>0$ there exists two constants $K:=K(r)>0$ and $\mu:=\mu(r)>0$ such that for any $y_0\in L^2(0,L)$ such that $\Vert y_0\Vert_{L^2(0,L)}\leq r$, the mild solution $y=y(t,x)$ to (\ref{nlkdv_sat}) satisfies
\begin{equation}
\label{local_exp_def}
\Vert y(t,.)\Vert_{L^2(0,L)}\leq K\Vert y_0\Vert_{L^2(0,L)}e^{-\mu t},\qquad \forall t\geq 0.
\end{equation}
\end{definition}

Following \cite{rosier2006global}, we first show that \eqref{nlkdv_sat} is semi-globally exponentially stable in $L^2(0,L)$. From this result, we will be able to prove the global uniform exponential stability of \eqref{nlkdv_sat}. To do that, we state and prove a technical lemma that allows us to bound the saturation function with a linear function as long as the initial condition is bounded. Then we separate our proof into two cases. The first one deals with the case $\omega=[0,L]$ and $\sat=\sath$, while the second one deals with the case $\omega\subseteq [0,L]$ whatever the saturation is. The tools to tackle these two cases are different. The goal of \startmodif the \stopmodif next three sections is to prove the following result
\begin{proposition}[Semi-global exponential stability]
\label{local_as_stab}
For all $y_0\in L^2(0,L)$ with $\Vert y_0\Vert_{L^2(0,L)}\leq r$, the system (\ref{nlkdv_sat}) is semi-globally exponentially stable in $L^2(0,L)$. 

Moreover, if $\omega=[0,L]$ and $\sat=\sath$, inequality \eqref{local_exp_def} holds with $K=1$ and $\mu$ can  be estimated as done in Theorem \ref{glob_as_stab}.
\end{proposition}

\subsection{Technical Lemma}

Before starting the proof of the Proposition \ref{local_as_stab}, let us state and prove the following lemma.
\begin{lemma}[Sector Condition]
\label{sat-l2-local}
Let $r$ be a positive value, a function $a:[0,L]\rightarrow \mathbb{R}$ satisfying \eqref{gain-control} and $k(r)$ defined by
 \begin{equation}
k(r)=\min\left\{\frac{u_0}{a_1 r},1 \right\}.
\end{equation}
\begin{itemize}
\item[(i)] Given $\sat=\sath$ and $s\in L^2(0,L)$ such that $\Vert s\Vert_{L^2(0,L)}\leq r$, we have
\begin{equation}
\Big( \sath(a(x)s(x))-k(r) a(x)s(x)\Big)s(x)\geq 0,\quad \forall x\in [0,L],
\end{equation}

\item[(ii)] Given $\sat=\satl$ and $s\in L^\infty(0,L)$ such that, for all $x\in [0,L]$, $|s(x)|\leq r$, we have
\begin{equation}
\Big(\satl(a(x)s(x))- k(r) a(x)s(x)\Big)s(x)\geq 0,\quad \forall x\in [0,L].
\end{equation}
\end{itemize}
\end{lemma}

\begin{proof}
(i) We first prove item (i) of Lemma \ref{sat-l2-local}. Two cases may occur
\begin{itemize}
\item[1.] $\Vert as\Vert_{L^2(0,L)}\geq u_0$;
\item[2.] $\Vert as\Vert_{L^2(0,L)}\leq u_0$.
\end{itemize}
The first case implies that, for all $x\in [0,L]$
$$
\sath(a(x)s(x))=\frac{a(x)s(x)}{\Vert as\Vert_{L^2(0,L)}}u_0.
$$
Thus, for all $x\in [0,L]$,
\begin{equation*}
\Big( \sath(a(x)s(x))-k(r) a(x)s(x)\Big)s(x)=a(x)s(x)^2\left(\frac{u_0}{\Vert as\Vert_{L^2(0,L)}}-k(r)\right).
\end{equation*}
Since
$$
\frac{u_0}{\Vert as\Vert_{L^2(0,L)}} \geq \frac{u_0}{a_1\Vert s\Vert_{L^2(0,L)}} \geq \frac{u_0}{a_1 r} \geq k(r),
$$
we obtain
$$
\Big(\sath(a(x)s(x))-k(r) a(x)s(x)\Big)s(x)\geq 0.
$$
Now, let us consider the case $\Vert as\Vert_{L^2(0,L)}\leq u_0$. We have, for all $x\in [0,L]$,
$$
\sath(a(x)s(x))=a(x)s(x),
$$
and then, for all $x\in [0,L]$,
\begin{equation*}
(\sath(a(x)s(x))-k(r) a(x)s(x))s(x)=a(x)s(x)^2(1-k(r))\geq 0.
\end{equation*}

(ii) We now deal with item (ii) of Lemma \ref{sat-l2-local}.

Let us pick $x\in [0,L]$ and consider the two following cases
\begin{itemize}
\item[1.] $|a(x)s(x)|\geq u_0$;
\item[2.] $|a(x)s(x)|\leq u_0.$
\end{itemize}
The first case implies either $a(x)s(x)\geq u_0$ or $a(x)s(x)\leq -u_0$.

Since these two possibilities are symmetric, we just deal with the case $a(x)\geq u_0$. We have
\begin{equation*}
\satl(a(x)s(x))=u_0,
\end{equation*}
and then
\begin{multline*}
\Big(\satl(a(x)s(x))-k(r)a(x)s(x)\Big)s(x)=u_0s(x)-k(r)a(x)s^2(x)\\
\geq \Big(u_0-k(r)a(x)r\Big)s(x)
\geq \left(u_0-\frac{u_0}{a_1 r}a(x)r\right)s(x)
\geq 0.
\end{multline*}
The second case implies that $\satl(a(x)s(x))=a(x)s(x),$
and then $
\Big(\satl(a(x)s(x)-k(r)a(x)s(x)\Big)s(x)=\Big(1-k(r)\Big)a(x)^2s(x)^2\geq 0.$
Thus it concludes the proof of the second item of the Lemma \ref{sat-l2-local}. 
\end{proof}
\subsection{Proof of Proposition \ref{local_as_stab} when $\omega=[0,L]$ and $\sat=\sath$}
\label{dead-zone-technique}

Now we are able to prove \startmodif Proposition \ref{local_as_stab} \stopmodif when $\omega=[0,L]$ and $\sat=\sath$. Let $r>0$ and $y_0\in L^2(0,L)$ be such that $\Vert y_0\Vert_{L^2(0,L)}\leq r$.

Multiplying \eqref{nlkdv_sat} by $y$, integrating with respect to $x$ on $(0,L)$ yields
\begin{equation}
\label{lyap-dz}
\frac{1}{2}\frac{d}{dt}\int_0^L |y(t,x)|^2dx\leq -\int_0^L \sath(ay(t,x))y(t,x)dx.
\end{equation}

Note that from \eqref{l2-dissipativity}, we get
\begin{equation}
\label{dissipativity-nl}
\Vert y\Vert_{L^2(0,L)} \leq \Vert y_0\Vert_{L^2(0,L)} \leq r.
\end{equation}
Thus, using Lemma \ref{sat-l2-local} and \eqref{lyap-dz}, it implies that
\begin{equation}
\frac{1}{2}\frac{d}{dt}\int_0^L |y(t,x)|^2dx\leq -\int_0^L k(r) a_0 |y(t,x)|^2dx.
\end{equation}
Applying the Gr\"onwall lemma leads to
\begin{equation}
\Vert y(t,.)\Vert_{L^2(0,L)}\leq e^{-\mu t}\Vert y_0\Vert_{L^2(0,L)}
\end{equation}
where $\mu$ is defined in the statement of Theorem \ref{glob_as_stab}. It concludes the proof of Proposition \ref{local_as_stab} when $\omega=[0,L]$ and when $\sat=\sath$. 

\begin{remark}
The constant $\mu$ depends on $u_0$, $r$ and $a_0$. Thus, although we have proven an exponential stability, the rapid stabilization is still an open question. Moreover, in the case $a(x)=a_0=a_1$ for all $x\in [0,L]$, which is the case where the gain is constant, we obtain that 
$$
\mu=\min\left\{a_0,\frac{u_0}{r}\right\}.
$$
\end{remark}

\subsection{Proof of Proposition \ref{local_as_stab} when $\omega \subseteq [0,L]$}


In this section, we have $\sat=\sath$ or $\sat=\satl$. We follow the strategy of \cite{rosier2006global} and \cite{cerpa2013control}. We use a  \startmodif contradiction argument\stopmodif. It is based on the following \startmodif unique continuation result\stopmodif.
\begin{theorem}[\cite{saut1987unique}]
\label{theorem-saut}
Let $u\in L^2(0,T;H^3(0,L))$ be a solution of
$$
u_t+u_x+u_{xxx}+uu_x=0
$$
such that
$$u(t,x)=0,\: \forall t\in (t_1,t_2),\: \forall x\in\omega,$$
with $\omega$ an open nonempty subset of $(0,L)$. Then
$$
u(t,x)=0,\: \forall t\in (t_1,t_2),\:\forall x\in (0,L).
$$
\end{theorem}

Moreover, the following lemma will be used.
\begin{lemma}[Aubin-Lions Lemma, \cite{simon1987compact}, Corollary 4]
\label{aubin-lions}
Let $X_0\subset X\subset X_1$ be three Banach spaces with $X_0$, $X_1$ reflexive spaces. Suppose that $X_0$ is compactly embedded in $X$ and $X$ is continuously embedded in $X_1$. Then $\lbrace h\in L^p(0,T;X_0)/\:  \dot{h}\in L^q(0,T;X_1)\rbrace$ embeds compactly in $L^p(0,T;X)$ for any $1<p,\: q<\infty$. 
\end{lemma}

Let us now start the proof of Proposition \ref{local_as_stab}. Let $r>0$ and $y_0\in L^2(0,L)$ be such that
\begin{equation}
\label{alternative-r}
\Vert y_0\Vert_{L^2(0,L)}\leq r.
\end{equation}
\startmodif As in the proof of Lemma \ref{global-estimation}, with multiplier techniques applied to \eqref{nlkdv_sat}, we obtain \stopmodif
\begin{equation}
\label{alternative1}
\Vert y(t,.)\Vert^2_{L^2(0,L)}=\Vert y_0\Vert^2_{L^2(0,L)}-\int_0^t |y_x(\sigma,0)|^2d\sigma -2\int_0^t\int_0^L \sat(ay)ydxd\sigma,\quad \forall t\in [0,T]
\end{equation}
and
\begin{equation}
\label{alternative-H1}
\Vert y\Vert^2_{L^2(0,T;H^1(0,L))}\leq \frac{8T+2L}{3}\Vert y_0\Vert^2_{L^2(0,L)}+\frac{TC}{27}\Vert y_0\Vert^4_{L^2(0,L)}.
\end{equation}
Moreover, multiplying \eqref{nlkdv_sat} by $(T-t)y$, we obtain after performing some integrations by parts
\begin{equation}
\label{rosier.3.13}
T\Vert y_0\Vert^2_{L^2(0,L)}\leq \int_0^T \int_0^L |y(t,x)|^2dxdt+\int_0^T (T-t)|y_x(t,0)|^2dt+2\int_0^T (T-t) \int_0^L \sat(ay)ydxdt.
\end{equation}

Note that, since $\sat$ is an odd function, \eqref{alternative1} implies that, for all $t\in [0,T]$
\begin{equation}
\label{dissipativity-alternative}
\Vert y(t,.)\Vert^2_{L^2(0,L)}\leq \Vert y_0\Vert^2_{L^2(0,L)}.
\end{equation}

From now on, we will separate the proof into two cases: $\sat=\sath$ and $\sat=\satl$.\\

\textbf{Case 1:} $\sat=\sath$. 

Using \eqref{l2-dissipativity}, we have, 
$$
\Vert y(T,.)\Vert_{L^2(0,L)}\leq r,
$$ 
and we can apply the first item of Lemma \ref{sat-l2-local}. The inequality \eqref{alternative1} becomes
\begin{equation}
\Vert y(T,.)\Vert^2_{L^2(0,L)}\leq \Vert y_0\Vert^2_{L^2(0,L)}-\int_0^T |y_x(t,0)|^2dt-2\int_0^T\int_0^L ak(r)y^2 dxdt.
\end{equation}

Let us state a claim that will be useful in the following.
\begin{claim}
\label{claim-statement}
For any $T>0$ and any $r>0$ there exists a positive constant $C_{14}=C_{14}(T,r)$ such that for any solution $y$ to \eqref{nlkdv_sat} with an initial condition $y_0\in L^2(0,L)$ such that $\Vert y_0\Vert_{L^2(0,L)}\leq r$, it holds that
\begin{equation}
\label{claim}
\Vert y_0\Vert^2_{L^2(0,L)}\leq C_{14}\left(\int_0^T |y_x(t,0)|^2dt+2\int_0^T\int_0^L k(r)a|y(t,x)|^2dxdt\right).
\end{equation}
\end{claim}
Let us assume Claim \ref{claim-statement} for the time being. Then \eqref{alternative1} implies
\begin{equation}
\Vert y(kT,.)\Vert^2_{L^2(0,L)}\leq \gamma^k\Vert y_0\Vert^2_{L^2(0,L)}\qquad \forall k\geq 0,\: \forall t\geq 0,
\end{equation}
where $\gamma\in (0,1)$. From \eqref{dissipativity-alternative}, we have $\Vert y(t,.)\Vert_{L^2(0,L)}\leq \Vert y(kT,.)\Vert_{L^2(0,L)}$ for $kT\leq t\leq (k+1)T$. Thus we obtain, for all $t\geq 0$,
\begin{equation}
\label{stability-absurd}
\Vert y(t,.)\Vert^2_{L^2(0,L)}\leq \frac{1}{\gamma}\Vert y_0\Vert_{L^2(0,L)}e^{\frac{\log \gamma}{T}t}.
\end{equation}
In order to prove Claim \ref{claim-statement}, since the solution to \eqref{nlkdv_sat} satisfies \eqref{rosier.3.13}, it is sufficient to prove that there exists some constant $C_{15}:=C_{15}(T,L)>0$ such that
\begin{equation}
\label{alternative2}
\int_0^T\int_0^L |y|^2dxdt\leq C_{15}\left(\int_0^T |y_x(t,0)|^2dt+2\int_0^T\int_0^L k(r)ay^2dxdt\right)
\end{equation}
provided that $\Vert y_0\Vert_{L^2(0,L)}\leq r$. We argue by contradiction to prove the existence of such a constant $C_{15}$.

Suppose \eqref{alternative2} fails to be true. Then, there exists a sequence of mild solutions $\lbrace y^n\rbrace_{n\in \mathbb{N}}\subseteq \mathcal{B}(T)$ of \eqref{nlkdv_sat} with
\begin{equation}
\label{alternative-initial-conditions}
\Vert y^n(0,.)\Vert_{L^2(0,L)}\leq r
\end{equation}
and such that
\begin{equation}
\label{alternative-limit}
\lim_{n\rightarrow +\infty}\frac{\Vert y^n\Vert^2_{L^2(0,T;L^2(0,L))}}{\int_0^T |y_x^n(t,0)|^2dt+2\int_0^T\int_0^L k(r)a(y^n)^2dxdt}=+\infty.
\end{equation}
Note that \eqref{alternative-initial-conditions} implies with \eqref{dissipativity-alternative} that
\begin{equation}
\label{alternative3}
\Vert y^n(t,.)\Vert_{L^2(0,L)}\leq r,\quad \forall t\in [0,T].
\end{equation}
Let $\lambda^n:=\Vert y^n\Vert_{L^2(0,T;L^2(0,L))}$ and $v^n(t,x)=\frac{y^n(t,x)}{\lambda^n}$. Notice that $\lbrace \lambda^n\rbrace_{n\in\mathbb{N}}$ is bounded, according to \eqref{alternative3}. \startmodif Hence, there exists a subsequence, that we continue to denote by $\lbrace\lambda^n\rbrace_{n\in\mathbb{N}}$ such that \stopmodif
$$
\lambda^n\rightarrow \lambda\geq 0\quad \text{as $n\rightarrow +\infty$}.
$$
Then $v^n$ fullfills
\begin{equation}
\label{alternative4}
\left\{
\begin{array}{l}
v^n_t+v^n_{x}+v^n_{xxx}+\lambda^n v^nv^n_x+\frac{\sath(a\lambda^n v^n)}{\lambda^n}=0,\\
v^n(t,0)=v^n(t,L)=v^n_x(t,L)=0,\\
\Vert v^n\Vert_{L^2(0,T;L^2(0,L))}=1,
\end{array}
\right.
\end{equation}
and, due to \eqref{alternative-limit}, we obtain
\begin{equation}
\label{alternative-limit-l2}
\int_0^T |v_x^n(t,0)|^2dt+2\int_0^T\int_0^L ak(r)(v^n)^2 dxdt\rightarrow 0\text{ as $n\rightarrow +\infty$.}
\end{equation}

It follows from \eqref{rosier.3.13} that $\lbrace v^n(0,.)\rbrace_{n\in\mathbb{N}}$ is bounded in $L^2(0,L)$. Note also that from \eqref{alternative-H1} $\lbrace v^n\rbrace_{n\in\mathbb{N}}$ is bounded in $L^2(0,T;H^1(0,L))$. Thus we see that $\lbrace v^nv^n_x\rbrace_{n\in\mathbb{N}}$ is a subset of $L^2(0,T;L^1(0,L))$. In fact,
\begin{equation}
\Vert v^nv^n_x\Vert_{L^2(0,T;L^1(0,L))}\leq \Vert v^n\Vert_{C(0,T;L^2(0,L))}\Vert v^n\Vert_{L^2(0,T;H^1(0,L))}.
\end{equation}
Moreover, we have that $\left\{\frac{\sath(a\lambda^n v^n)}{\lambda^n}\right\}_{n\in\mathbb{N}}$ is a bounded sequence in $L^2(0,T;L^2(0,L))$. Indeed, from Lemma \ref{lipschitz-satl2}
\begin{equation}
\left\Vert \frac{\sath(a\lambda^n v^n)}{\lambda^n}\right\Vert_{L^2(0,T;L^2(0,L))}\leq 3\Vert av^n \Vert_{L^2(0,T;L^2(0,L))}\leq 3a_1\sqrt{L}\Vert v^n\Vert_{L^2(0,T;H^1(0,L))}.
\end{equation}
Thus $\lbrace v^nv^n_x\rbrace_{n\in\mathbb{N}}$ and $\left\{\frac{\sath(a\lambda^n v^n)}{\lambda^n}\right\}_{n\in\mathbb{N}}$  are also subsets of $L^2(0,T;H^{-2}(0,L))$ since $L^2(0,L)\subset L^1(0,L)\subset H^{-1}(0,L)\subset H^{-2}(0,L)$. Combined with \eqref{alternative4} it implies that $\lbrace v_t^n\rbrace_{n\in\mathbb{N}}$ is a bounded sequence in $L^2(0,T;H^{-2}(0,L))$. Since $\lbrace v^n\rbrace_{n\in\mathbb{N}}$ is a bounded sequence of $L^2(0,T;H^1(0,L))$, then we get with Lemma \ref{aubin-lions} that a subsequence of $\lbrace v^n\rbrace_{n\in\mathbb{N}}$ also denoted by $\lbrace v^n\rbrace_{n\in\mathbb{N}}$ converges strongly in $L^2(0,T;L^2(0,L))$ to a limit $v$. Moreover, with the last line of \eqref{alternative4}, it holds that $\Vert v\Vert_{L^2(0,T;L^2(0,L))}=1$.

Therefore, having in mind \eqref{alternative-limit-l2}, we get
\begin{equation}
\Vert v_x(.,0)\Vert^2_{L^2(0,T)}+\int_0^T\int_0^L ak(r)v^2 dxdt\leq \liminf_{n\rightarrow +\infty} \left\{\Vert v^n_x(.,0)\Vert^2_{L^2(0,T)}+\int_0^T \int_0^L ak(r)(v^n)^2dxdt\right\}=0.
\end{equation}
Thus,
\begin{equation}
ak(r)v^2(t,x)=0,\: \forall x\in [0,L],\forall t\in (0,T),\: \text{ and }\: v_x(t,0)=0,\: \forall t\in (0,T).
\end{equation}
and therefore
\begin{equation}
\label{unique-continuation}
v(t,x)=0,\: \forall x\in \omega,\forall t\in (0,T),\: \text{ and }\: v_x(t,0)=0,\: \forall t\in (0,T).
\end{equation}


We obtain that the limit function $v$ satisfies
\begin{equation}
\label{kdv-unique-continuation}
\left\{
\begin{array}{l}
v_t+v_x+v_{xxx}+\lambda vv_x=0,\\
v(t,0)=v(t,L)=v_x(t,L)=0,\\
\Vert v\Vert_{L^2(0,T;L^2(0,L))}=1,
\end{array}
\right.
\end{equation}
with $\lambda\geq 0$. Let us consider $u:=v_t$ which satisfies
\begin{equation}
\label{kdv-unique-continuation-u}
\left\{
\begin{array}{l}
u_t+u_x+u_{xxx}+\lambda v_xu+\lambda vu_x=0,\\
u(t,0)=u(t,L)=u_x(t,L)=0,
\end{array}
\right.
\end{equation}
with $u(0,.)=-v^\prime(0,.)-v^{\prime\prime\prime}(0,.)-\lambda v(0,.)v^\prime(0,.)\in H^{-3}(0,L)$
and
$$
u(t,x)=0,\: \forall x\in \omega,\forall t\in (0,T),\: \text{ and }\: u_x(t,0)=0,\: \forall t\in (0,T).
$$
Let us recall the following result. 
\begin{lemma}[\cite{pazoto2005localizeddamping}, Lemma 3.2]
There exists a positive value $C_{16}(T,r)>0$ such that for any solution $u$ to \eqref{kdv-unique-continuation-u} where $v$ is solution to \eqref{kdv-unique-continuation}, it holds 
\begin{equation}
\Vert u_x(.,0)\Vert^2_{L^2(0,T)}+\Vert u(0,.)\Vert^2_{H^{-3}(0,L)}\geq C_{16}\Vert u(0,.)\Vert^2_{L^2(0,L)}.
\end{equation}
\end{lemma}
Applying the result of this lemma, we get $u(0,.)\in L^2(0,L)$ and therefore $u=v_t\in\mathcal{B}(T)$. Since $v,v_t\in L^2(0,T;H^1(0,L))$ and $v\in C([0,T];H^1(0,L))$, we can conclude that $vv_x\in L^2(0,T;L^2(0,L))$. In this way, $v_{xxx}=-v_t-v_x-\lambda vv_x\in L^2(0,T;L^2(0,L))$ and therefore $v\in L^2(0,T;H^3(0,L))$. Finally, using Theorem \ref{theorem-saut}, we obtain
$$
v(t,x)=0,\quad \forall x\in [0,L],\: t\in [0,T].
$$
Thus we get a contradiction with $\Vert v\Vert_{L^2(0,T;L^2(0,L))}=1$. It concludes the proof of Claim \ref{claim-statement} and thus Lemma \ref{local_as_stab} in the case where $\sat=\sath$. \\

\textbf{Case 2:} $\sat=\satl$. 

Following the same strategy than before, we write the following claim. 
\begin{claim}
\label{claim-statement1}
For any $T>0$ and any $r>0$, there exists a positive constant $C_{17}=C_{17}(T,r)$ such that for any mild solution $y$ to \eqref{nlkdv_sat} with an initial condition $y_0\in L^2(0,L)$ such that $\Vert y_0\Vert_{L^2(0,L)}\leq r$, it holds that
\begin{equation}
\label{claim1}
\Vert y_0\Vert^2_{L^2(0,L)}\leq C_{17}\left(\int_0^T |y_x(t,0)|^2dt+2\int_0^T \int_0^L \satl(ay(t,x))y(t,x)dtdx\right).
\end{equation}
\end{claim}
If Claim \ref{claim-statement1} holds, we obtain also \eqref{stability-absurd} for a suitable choice of $\gamma$ and we end the proof of Lemma \ref{local_as_stab} when $\sat=\satl$. \startmodif Due to \stopmodif \eqref{rosier.3.13}, we see that in order to prove Claim \ref{claim-statement1}, it is sufficient to obtain the existence of $C_{18}>0$ such that
\begin{equation}
\label{claim1-contradiction}
\int_0^T \int_0^L  |y(t,x)|^2dtdx\leq C_{18}\left(\int_0^T |y_x(t,0)|^2dt+2\int_0^T \int_0^L \satl(ay(t,x))y(t,x)dtdx\right).
\end{equation}
We argue by contradiction to prove \eqref{claim1-contradiction}. Then, we assume that there exists a sequence of mild solutions $\lbrace y^n\rbrace_{n\in\mathbb{
N}}\subseteq \mathcal{B}(T)$ to \eqref{nlkdv_sat} with
\begin{equation}
\label{alternative-initial-conditions1}
\Vert y^n(0,.)\Vert_{L^2(0,L)}\leq r
\end{equation}
and such that
\begin{equation}
\label{alternative-limit1}
\lim_{n\rightarrow +\infty}\frac{\Vert y^n\Vert^2_{L^2(0,T;L^2(0,L))}}{\int_0^T |y_x^n(t,0)|^2dt+2\int_{0}^T \int_0^L \satl(ay^n(t,x))y^n(t,x)dtdx}=+\infty.
\end{equation}
Note that \eqref{alternative-initial-conditions1} implies with \eqref{dissipativity-alternative} that
\begin{equation}
\label{alternative31}
\Vert y^n(t,.)\Vert_{L^2(0,L)}\leq r,\quad \forall t\in [0,T].
\end{equation}

Note that  we have, from \eqref{alternative-r} and \eqref{alternative-H1}
$$\Vert y^n\Vert^2_{L^2(0,T;H^1(0,L))}\leq \beta,$$
where
$$
\beta:=\frac{8T+2L}{3}r^2+\frac{TC}{27}r^4.
$$
Moreover, due to Poincar\'e inequality and the left Dirichlet boundary condition of \eqref{nlkdv_sat}, we obtain
\begin{equation}
\sup_{x\in [0,L]} |y^n(t,x)|\leq \sqrt{L}\Vert y^n(t,.)\Vert_{H^1(0,L)},\quad \forall t\in [0,T].
\end{equation}
Thus, we see that
\begin{equation}
\label{omega1}
\int_0^T |y^n(t,x)|^2dt\leq L\Vert y^n\Vert^2_{L^2(0,T;H^1(0,L))}\leq L\beta.
\end{equation}
Now let us consider $\Omega_i\subset [0,T]$ defined as follows
\begin{equation}
\Omega_i=\left\{ t\in [0,T],\: \sup_{x\in [0,L]}|y(t,x)| >i\right\}. 
\end{equation}
In the following, we will denote by $\Omega_i^c$ its complement. It is defined by
\begin{equation}
\Omega_i^c=\left\{ t\in [0,T],\: \sup_{x\in [0,L]}|y(t,x)| \leq i\right\}. 
\end{equation}
Since the function $t\mapsto \sup_{x\in [0,L]}|y^n(t,x)|^2$ is a nonnegative function, we have
\begin{equation}
\label{omega2}
\int_0^T \sup_{x\in [0,L]}|y^n(t,x)|^2dt \geq \int_{\Omega_i} \sup_{x\in [0,L]}|y^n(t,x)|^2dt\geq i^2 \nu(\Omega_i),
\end{equation}
where $\nu(\Omega_i)$ denotes the Lebesgue measure of $\Omega_i$. Therefore, with \eqref{omega1}, we obtain
\begin{equation}
\nu(\Omega_i)\leq \frac{L\beta}{i^2}.
\end{equation}
We deduce from the previous equation that
\begin{equation}
\label{lebesgue-measure-satloc}
\max\left(T-\frac{L\beta}{i^2},0\right) \leq \nu(\Omega_i^c)\leq T.
\end{equation}
Moreover, with the second item of Lemma \ref{sat-l2-local}, we have, for all $i\in\mathbb{N}$,
\begin{eqnarray}
\label{sector-condition-satloc}
\int_0^T \int_0^L \satl(ay^n)y^ndtdx= &&\int_{\Omega_i}\int_0^L \satl(ay^n)y^ndtdx+\int_{\Omega_i^c}\int_0^L \satl(ay^n)y^ndtdx\nonumber \\
\geq && \int_{\Omega_i^c}\int_0^L \satl(ay^n)y^ndtdx\nonumber \\
\geq && \int_{\Omega_i^c} \int_0^L ak(i)(y^n)^2dtdx.
\end{eqnarray}

Let $\lambda^n:=\Vert y^n\Vert_{L^2(0,T;L^2(0,L))}$ and $v^n(t,x)=\frac{y^n(t,x)}{\lambda^n}$. Notice that $\lbrace\lambda^n\rbrace_{n\in\mathbb{N}}$ is bounded, according to \eqref{alternative31}. Hence, there exists a subsequence, that we continue to denote by $\lbrace\lambda^n\rbrace_{n\in\mathbb{N}}$ such that
$$
\lambda^n\rightarrow \lambda\geq 0,\text{ as $n\rightarrow +\infty$}.
$$
Then, $v^n$ fullfills
\begin{equation}
\label{alternative41}
\left\{
\begin{array}{l}
v^n_t+v^n_{x}+v^n_{xxx}+\lambda^n v^nv^n_x+\frac{\satl(a\lambda^n v^n)}{\lambda^n}=0,\\
v^n(t,0)=v^n(t,L)=v^n_x(t,L)=0,\\
\Vert v^n\Vert_{L^2(0,T;L^2(0,L))}=1,
\end{array}
\right.
\end{equation}
and, due to \eqref{alternative-limit1},
$$
\int_0^T |v_x^n(t,0)|^2dt+2\int_
0^T \int_0^L \frac{\satl(a\lambda^n v^n)}{\lambda^n}v^n dtdx\rightarrow 0\text{ as $n\rightarrow +\infty$.}
$$
Moreover, due to \eqref{sector-condition-satloc}, we have, for all $i\in\mathbb{N}$,
\begin{equation}
\label{convergence-absurde-satloc}
\int_0^T |v_x^n(t,0)|^2dt+2\int_{\Omega_i^c} \int_0^L ak(i)(v^n)^2dtdx\rightarrow 0\text{ as $n\rightarrow +\infty$.}
\end{equation}

Note that from Lemma \ref{lipschitz-satl2},
\begin{equation}
\left\Vert \frac{\satl(a\lambda^n v^n)}{\lambda^n}\right\Vert_{L^2(0,T;L^2(0,L))}\leq 3a_1\sqrt{L}\Vert v^n\Vert_{L^2(0,T;H^1(0,L)}.
\end{equation}
and therefore the sequence $\left\{ \frac{\satl(a\lambda^n v^n)}{\lambda^n}\right\}_{n\in\mathbb{N}}$ is a subset of $L^2(0,T;L^2(0,L))$. In addition, $\left\{ v^nv^n_x\right\}_{n\in\mathbb{N}}$ is a bounded sequence of $L^2(0,T;L^1(0,L))$. Note that $L^1(0,L)\subset L^2(0,L)\subset H^{-2}(0,L)$, thus $\left\{ \frac{\satl(a\lambda^n v^n)}{\lambda^n}\right\}_{n\in\mathbb{N}}$ and $\left\{ v^nv^n_x\right\}_{n\in\mathbb{N}}$ are bounded sequences of $L^2(0,T;H^{-2}(0,L))$. Since $v^n_t=-v^n_x-v^n_{xxx}-\lambda^n v^nv^n_x-\frac{\satl(a\lambda^n v^n)}{\lambda^n}$, we know that $\lbrace v^n_t\rbrace_{n\in\mathbb{N}}$ is a subset of $L^2(0,T;H^{-2}(0,L))$. Since $\lbrace v^n\rbrace_{n\in\mathbb{N}}$ is a subset of $L^2(0,T;H^1(0,L))$, we obtain from Lemma \ref{aubin-lions} that $\lbrace v^n\rbrace_{n\in \mathbb{N}}$ \startmodif converges strongly to a function $v$ \stopmodif in $L^2(0,T;L^2(0,L))$. Futhermore, with \eqref{convergence-absurde-satloc} and due to the non-negativity of $k$, we have, for all $i\in\mathbb{N}$, 
\begin{equation}
\label{unique-continuation2}
ak(i)v(t,x)=0,\: \forall x\in [0,L],\forall t\in \Omega_i^c,\: \text{ and }\: v_x(t,0)=0,\: \forall t\in (0,T).
\end{equation}
Thus, since for all $i\in\mathbb{N}$, $k(i)$ is strictly positive,  we have
\begin{equation}
\label{unique-continuation3}
v(t,x)=0,\: \forall x\in \omega,\forall t\in \Omega_i^c,\: \text{ and }\: v_x(t,0)=0,\: \forall t\in (0,T).
\end{equation}
We obtain
\begin{equation}
\label{unique-continuation4}
v(t,x)=0,\: \forall x\in \omega,\forall t\in \bigcup_{i\in\mathbb{N}}\Omega_i^c ,\: \text{ and }\: v_x(t,0)=0,\: \forall t\in (0,T).
\end{equation}
Since, with \eqref{lebesgue-measure-satloc}, we know that $\nu\left(\bigcup_{i\in\mathbb{N}}\Omega_i^c\right)=T$, we get that, for almost every $t\in [0,T]$,
\begin{equation}
\label{unique-continuation5}
v(t,x)=0,\: \forall x\in \omega ,\: \text{ and }\: v_x(t,0)=0.
\end{equation}
We obtain that $v$ fullfills 
\begin{equation}
\label{kdv-uc-satl}
\left\{
\begin{array}{l}
v_t+v_x+v_{xxx}+\lambda vv_x=0,\\
v(t,0)=v(t,L)=v_x(t,L)=0,\ \startmodif, \stopmodif\
\Vert v\Vert_{L^2(0,T;L^2(0,L))}=1.
\end{array}
\right.
\end{equation}
Thus $v$ is a solution to a Korteweg-de Vries equation. In particular, it belongs to $\mathcal{B}(T)$ and is consequently in $C(0,T;L^2(0,L))$. Therefore, \eqref{unique-continuation5} becomes
\begin{equation}
\label{unique-continuation6}
v(t,x)=0,\: \forall x\in \omega,\forall t\in [0,T],\: \text{ and }\: v_x(t,0)=0,\: \forall t\in (0,T).
\end{equation}

We are in the same situation \startmodif as \stopmodif \eqref{kdv-unique-continuation}. Therefore we obtain once again a contradiction. We can conclude that Claim 2 is true. It concludes the proof of Lemma \ref{local_as_stab} when $\sat=\satl$ and completes the proof of Proposition \ref{local_as_stab}. \null\hfill$\Box$
 
\begin{remark}
Since the strategy followed in the last section is to argue by contradiction, we cannot estimate the exponential rate $\mu$. However, such a proof allows us to prove the local exponential stability of the solution whatever the saturation $\sat$ is.
\end{remark}

\subsection{Proof of Theorem \ref{glob_as_stab}}

\label{section_astuce}

We are now in position to prove Theorem \ref{glob_as_stab}, following \cite{rosier2006global}.
By \startmodif Proposition \ref{local_as_stab}\stopmodif, there exists  $\mu^{\star}>0$ such that if
\begin{equation}
\Vert \tilde{y}_0\Vert_{L^2(0,L)}\leq 1,
\end{equation}
then the corresponding solution $\tilde{y}$ to (\ref{nlkdv_sat}) satisfies
\begin{equation}
\label{1-nl-stab-lyap}
\Vert \tilde{y}(t,.)\Vert_{L^2(0,L)}\leq K_1\Vert \tilde{y}_0\Vert_{L^2(0,L)}e^{-\mu^{\star}t}\qquad \forall t\geq 0,
\end{equation}
for some constants $K_1\geq 1$ which depends only on $\Vert \tilde{y}_0\Vert_{L^2(0,L)}$. In addition, for a given $r>0$, there exist two constants $K_r>0$ and $\mu_r>0$ such that if $\Vert y_0\Vert_{L^2(0,L)}\leq r$, then any mild solution $y$ to (\ref{nlkdv_sat}) satisfies
\begin{equation}
\Vert y(t,.)\Vert_{L^2(0,L)}\leq K_r \Vert y_0\Vert_{L^2(0,L)}e^{-\mu_r t}\qquad \forall t\geq 0.
\end{equation}
Consequently, setting $T_r:=\mu_r^{-1}\ln(rK_r)$, we have
$$
\Vert y_0\Vert_{L^2(0,L)}\leq r\Rightarrow \Vert y(T_r,.)\Vert_{L^2(0,L)}\leq 1.
$$
Therefore, using \eqref{1-nl-stab-lyap}, we obtain
\begin{equation}
\begin{array}{rcl}
\Vert y(t,.)\Vert_{L^2(0,L)} &\leq& K_1\Vert y(T_r,.)\Vert_{L^2(0,L)}e^{-\mu^{\star}(t-T_r)}\qquad \forall t\geq T_r,\\
&\leq& K_1K_r\Vert y_0\Vert_{L^2(0,L)}e^{\mu^{\star}T_r}e^{-\mu^{\star}t}\qquad \forall t\geq 0.
\end{array}
\end{equation}
Thus it concludes the proof of Theorem \ref{glob_as_stab}. \hfill \hspace*{1pt}\hfill $\Box$

\startmodif
\begin{remark}
As it has been noticed in Remark \ref{remark-gkdv}, the same result than Theorem \ref{glob_as_stab} can be obtained for \eqref{gkdv-rosier} following the strategy of \cite{rosier2006global}. Note that in \cite{rosier2006global}, a stabilization in $H^3(0,L)$ is obtained. The authors used a similar strategy than the one described in Remark \ref{remark-gkdv}. Hence, it seems harder to obtain such a result for \eqref{gkdv-saturated}, since the saturation introduces some non-smoothness.
\end{remark}
\stopmodif
\section{Simulations}

\label{sec_simu}

In this section we provide some numerical simulations showing the effectiveness of our control design. In order to discretize our KdV equation, we use a finite difference scheme inspired by \cite{nm_KdV}. The final time is denoted $T_{final}$. We choose $(N_x+1)$ points to build a uniform spatial discretization of the interval $[0,L]$ and $(N_t+1)$ points to build a uniform time discretization of the interval $[0,T_{final}]$. We pick a space step defined by $dx=L/N_x$ and a time step defined by $dt=T_{final}/N_t$. We approximate the solution with the following notation $y(t,x)\approx Y^i_j$, where $i$ denotes the time and $j$ the space discrete variables. 

Some used approximations of the derivative are given by
\begin{equation}
\mathcal{D}_-y=\frac{Y^i_j-Y^i_{j-1}}{dx}
\end{equation}
and
\begin{equation}
\mathcal{D}_+y=\frac{Y^i_{j+1}-Y^i_{j}}{dx}.
\end{equation}
As in \cite{nm_KdV}, we choose the numerical scheme $
y_x(t,x)\approx\frac{1}{2}(\mathcal{D}_++\mathcal{D}_-)(Y^i_j):=\mathcal{D}(Y^i_j)$ and $
y_t(t,x)\approx \frac{Y^{i+1}_j-Y^i_j}{dt}.$
For the other differentiation operator, we use
$y_{xxx}(t,x)\approx \mathcal{D}_+\mathcal{D}_+\mathcal{D}_-(Y^i_j)$.

Let us introduce a matrix notation. Let us
consider the matrices $D_-, D_+, D \in\mathbb{R}^{N_x\times N_x}$ given by
\begin{equation}
D_-=\frac{1}{dx}\begin{bmatrix}
1 & 0 & \ldots & \ldots & 0\\
-1 & 1 & \ddots & & \vdots\\
0 & \ddots & \ddots & \ddots & \vdots\\
\vdots & \ddots & \ddots & 1 & 0\\
0 & \ldots & 0 & -1 & 1
\end{bmatrix},\: D_+=\frac{1}{dx}\begin{bmatrix}
-1 & 1 & 0 & \ldots & 0\\
0 & -1 & 1 & \ddots & \vdots\\
\vdots & \ddots & \ddots & \ddots & 0 \\
\vdots &  & \ddots & -1 & 1\\
0 & \ldots &\ldots & 0 & -1
\end{bmatrix}
\end{equation}
\begin{equation}
D:=\frac{1}{2}(D_++D_-)
\end{equation} 
and let us define $\mathcal{A}=D_+D_+D_-+D$, and $\mathcal{C}=\mathcal{A}+dtI$ where $I$ is the identity matrix in $\mathcal{M}_{N_x\times N_x}(\mathbb{R})$. Note that we choose this forward difference approximation in order to obtain a positive definite matrix $\mathcal{C}$.  

Moreover, for each discrete time $i$, we denote $Y^i:=\begin{bmatrix}
Y^i_1 & Y^i_2 & \ldots & Y_{N_x+1}^i
\end{bmatrix}^{\top}$.

Thus, inspired by \cite{nm_KdV}, we consider a completely implicit numerical scheme for the approximation of the nonlinear problem \eqref{nlkdv_sat} which reads as follows:
\begin{equation}
\label{numerical_scheme}
\left\{
\begin{array}{l}
\displaystyle \frac{Y_j^{i+1}-Y_j^{i}}{dt}+(\mathcal{A}Y^{i+1})_j+\frac{1}{2}\left(D[(Y^{i+1})^2]\right)_j+\sat(a_\delta Y_j^{i+1})=0,\quad j=1,\ldots N_x,\\
 Y^{i}_1=Y^i_{N_x+1}=Y^i_{N_x}=0,\\
 Y^1=\int_{x_{j-\frac{1}{2}}}^{x_{j+\frac{1}{2}}}y_0(x)dx,
\end{array}\right.
\end{equation}
where $x_j=(j+\frac{1}{2})dx$, $x_j=jdx$ and $Y^1$ denotes the discretized version of the initial condition $y_0(x)$. Note that $a_\delta$ is the approximation of the damping function $a=a(x)$ and is given by $a_\delta=\left( a_j\right)_{J=1}^{N_x}\in\mathbb{R}^{N_x}$, where each components $a_j$ is defined by $a_j:=\int_{x_{j-\frac{1}{2}}}^{x_{j+\frac{1}{2}}} a(x)dx$. 

Since we have the nonlinearities $yy_x$ and $\sat(ay)$, we use an iterative Newton fixed-point method to solve the nonlinear system 
$$\mathcal{C}Y^{i+1}=Y^i-dt\frac{1}{2}D(Y^{i+1})^2-dt\sat\left(a_\delta Y^{i+1}\right).$$
With $N_{iter}=5$, which denotes the number of iterations of the fixed point method, we get good approximations of the solutions. Note that for sufficiently large $N_{iter}$ the solutions can be approximated with this fixed-point method.

Given $Y^1$ satisfying \eqref{numerical_scheme}, the following is the structure of the algorithm used in our simulations. \\

\fbox{\begin{minipage}{15cm}
\textbf{For $i=1:N_t$}\\
\begin{itemize}
\item[$\bullet$] $Y_{1}^{i}=Y^{i}_{N_x}=Y^{i}_{N_{x}+1}=0$;
\item[$\bullet$] Setting $J(1)=Y^i$, for all $k\in\lbrace 1,\ldots, N_{iter}\rbrace$, solve
$$J(k+1)=C^{-1}(Y^i-dt\frac{1}{2}D(J(k))^2-dt\sat(a_\delta J(k)))$$\\
Set $Y^{i+1}=J(N_{iter})$
\end{itemize}
\textbf{end}
\end{minipage}}\\

In order to illustrate our theoretical results, we perform some simulations with $L= 2\pi$, for which we know that the linearized KdV equation is not asymptotically stable. To be more specific, letting $y_0(x)=1-\cos(x)$ and $f=0$, it holds that the energy $\Vert y\Vert^2_{L^2(0,L)}$ of the linearized equation \eqref{lkdv} remains constant for all $t\geq 0$. Let us perform a simulation of \eqref{nlkdv_sat} with these parameters. 

We first simulate our system in the case where the damping is not localized. We use the saturation function $\sath$. Given $a_0=1$, $T_{final}=6$ and $L=2\pi$, Figure \ref{figure1} shows the solution to \eqref{nlkdv}, denoted by $y_w$, with the unsaturated control $f=a_0y_w$ and starting from $y_0$. Figure \ref{figure2} illustrates the simulated solution with the same initial condition and a saturated control $f=\sath(a_0y)$ where $u_0=0.5$. Figure \ref{figure3} gives the evolution of the control with respect to the time and the space. We check in Figures \ref{figure1} and \ref{figure2} that the solution to \eqref{nlkdv_sat} converges to $0$ with the unsaturated and the saturated controls as proven in Theorem \ref{glob_as_stab}. 

The evolution of the $L^2$-energy of the solution in these two cases is given by Figure \ref{figure4}. With $\Vert y_0\Vert_{L^2(0,L)}:=3.07$ and the values of $u_0$, $a_0$ and $a_1$, the value $\mu$ is computed numerically following the formula \eqref{formula-mu} given in Theorem \ref{glob_as_stab}. It is is equal to $\mu=0.3257$. We deduce from the second point of Theorem \ref{glob_as_stab} that the energy function $\Vert y\Vert^2_{L^2(0,L)}$ converges exponentially to $0$ with an explicit decay rate given by $\mu$ as stated in Theorem \ref{glob_as_stab}. 

\begin{figure}[H]
   \begin{minipage}[c]{.50\linewidth}
      \includegraphics[scale=0.6]{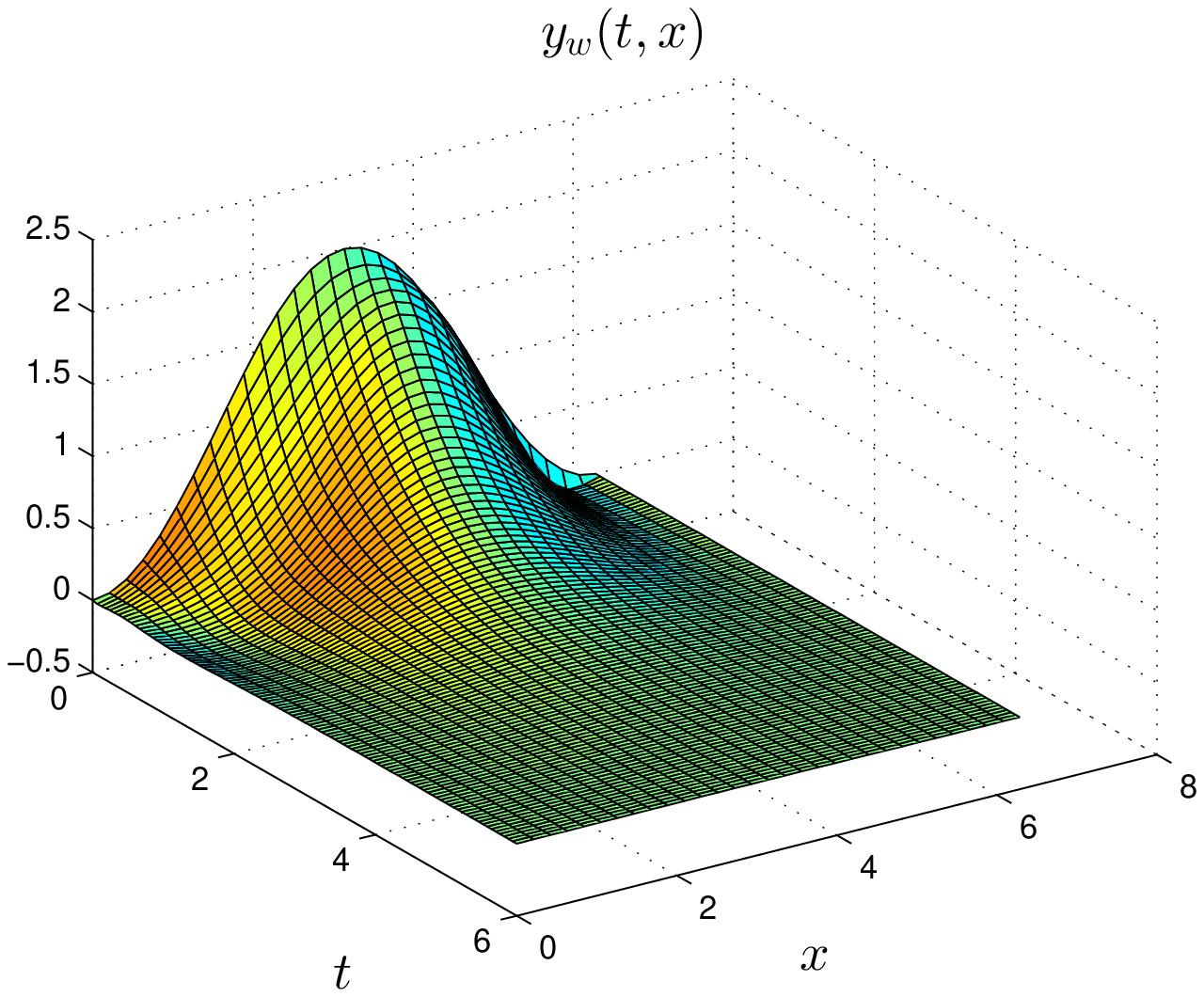}
      \caption{Solution $y_w(t,x)$ with the control $f=a_0y_w$ where $\omega=[0,L]$}
      \label{figure1}
   \end{minipage} \hfill
   \begin{minipage}[c]{.46\linewidth}
      \includegraphics[scale=0.6]{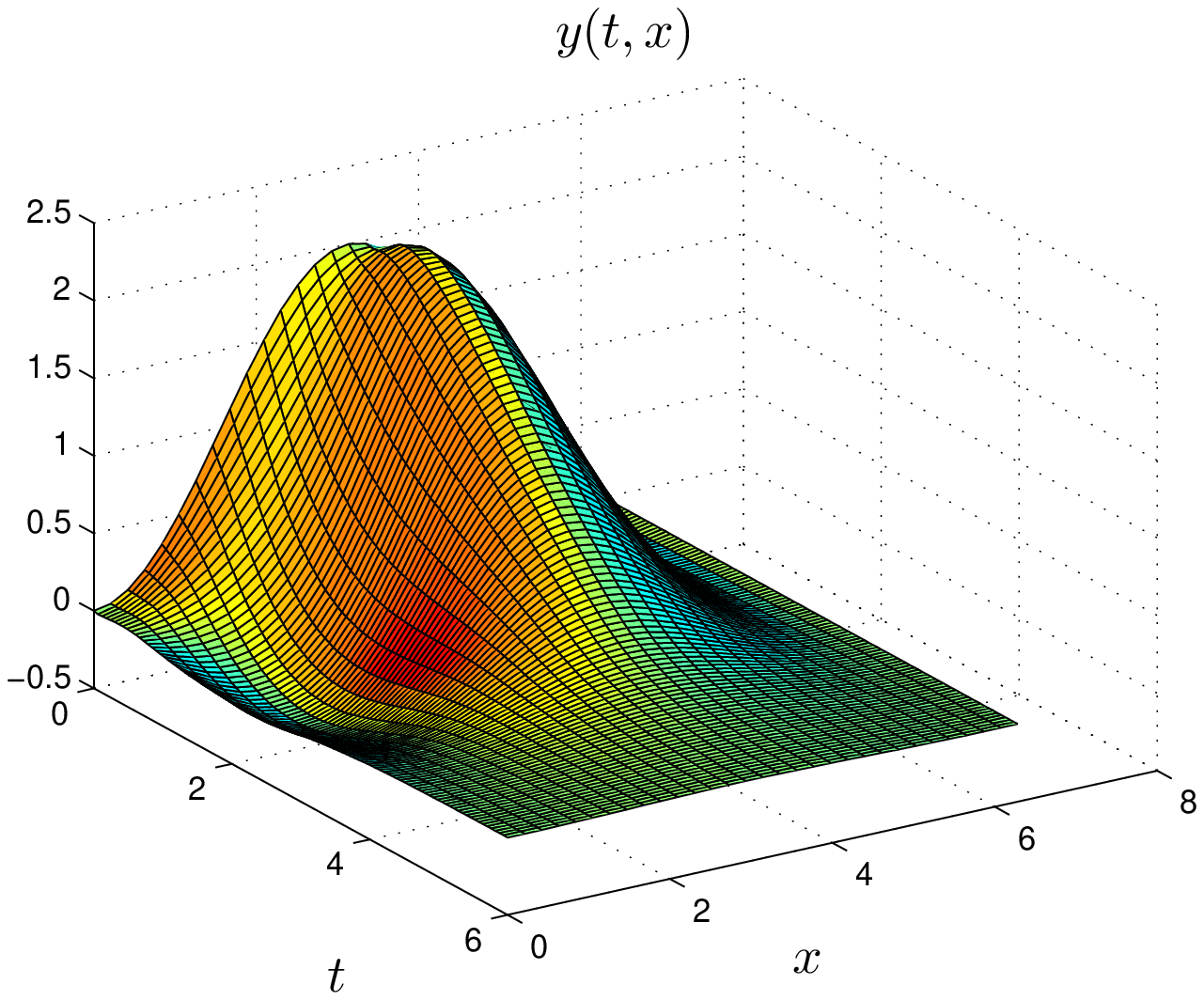}
      \caption{Solution $y(t,x)$ with the control $f=\sath(a_0y)$ where $\omega=[0,L]$, $u_0=0.5$}
      \label{figure2}
   \end{minipage}\hfill
   \end{figure}
\begin{figure}[H]   
   \begin{minipage}[c]{.46\linewidth}
      \includegraphics[scale=0.6]{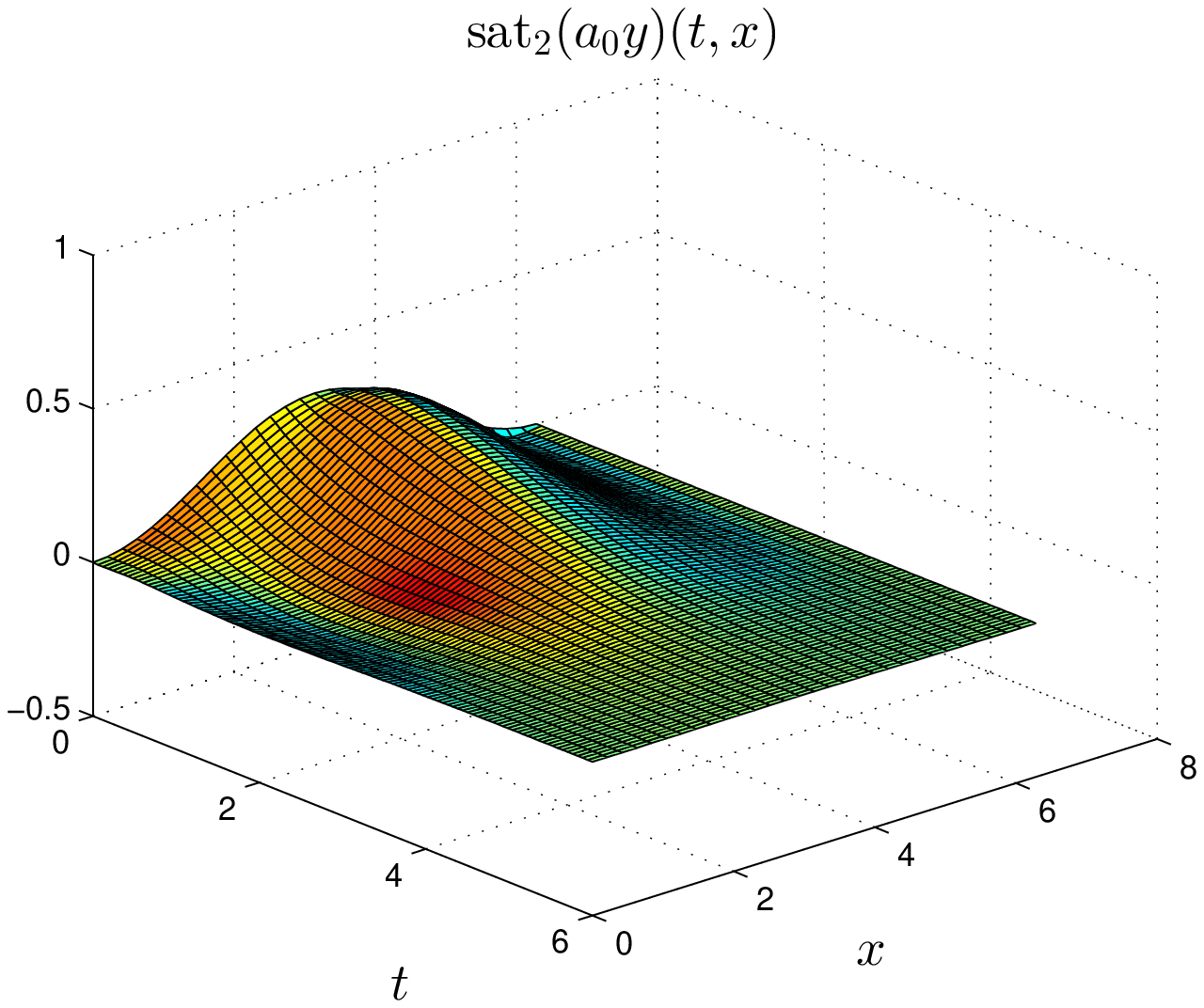}
\caption{Control $f=\sath(a_0 y)(t,x)$ where $\omega=[0,L]$, $u_0=0.5$}
\label{figure3}
   \end{minipage}\hfill
   \begin{minipage}[c]{.46\linewidth}
      \includegraphics[scale=0.6]{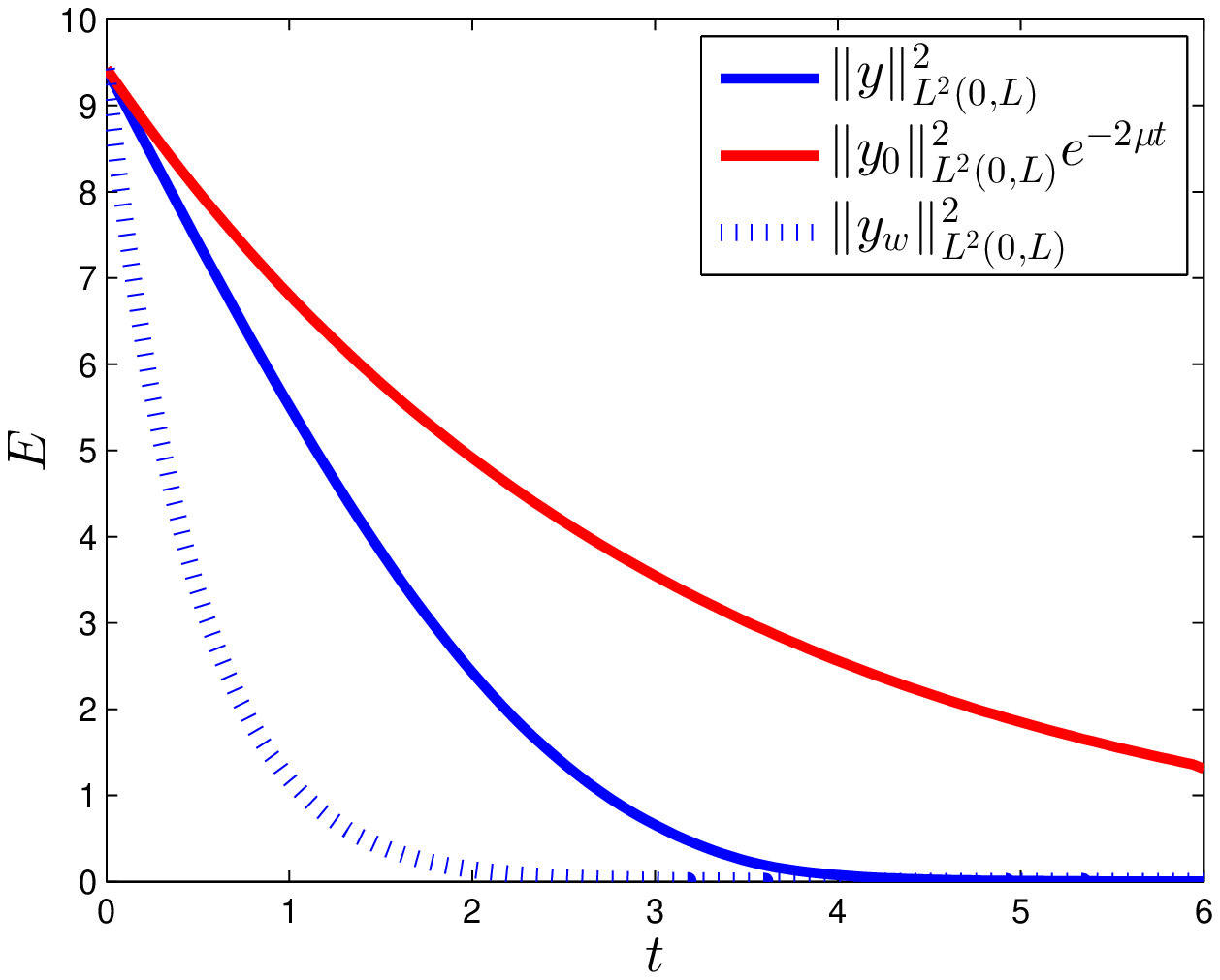}
      \caption{Blue: Time evolution of the energy function $\Vert y\Vert^2_{L^2(0,L)}$ with a saturation $u_0=0.5$ and $a_0=1$. Red: Time evolution of the theoritical energy $\Vert y_0\Vert^2_{L^2(0,L)}e^{-2\mu t}$. Dotted line: Time evolution of the solution without saturation $y_w$ and $a_0=1$.}
      \label{figure4}
   \end{minipage}
\end{figure}

We now focus on the case where the damping is localized. We close the loop with the saturated controller $f=\satl(ay)$ where $a$ is defined by $a(x)=a_0=1$, for all $x\in\omega:=\left[\frac{1}{3}L,\frac{2}{3}L\right].$

Given $T_{final}=6$, Figure \ref{figure7} shows the simulated solution of \eqref{nlkdv}, denoted by $y_w$, with a localized control that is not saturated and starting from $y_0$. Figure \ref{figure8} illustrates the simulated solution to \eqref{nlkdv_sat} with the same initial condition, but with a localized saturated control whose saturation level is given by $u_0=0.5$. We check, in Figures \ref{figure7} and \ref{figure8}, that the mild solution to \eqref{nlkdv_sat} converges to $0$ as stated in Theorem \ref{glob_as_stab}. Moreover, Figure \ref{figure9} gives the evolution of the control with respect to the time and the space. 

The evolution of the $L^2$-energy of the solution in these two last cases is given by Figure \ref{figure10}. We can see that the energy function $\Vert y\Vert^2_{L^2(0,L)}$ converges exponentially to $0$ as stated in Proposition \ref{local_as_stab}. However, in contrary with the case $\sat=\sath$ and $\omega=[0,L]$, we cannot have an estimation of the decay rate since our proof is based on a contradiction argument. 

\begin{figure}[H]
   \begin{minipage}[c]{.46\linewidth}
      \includegraphics[scale=0.6]{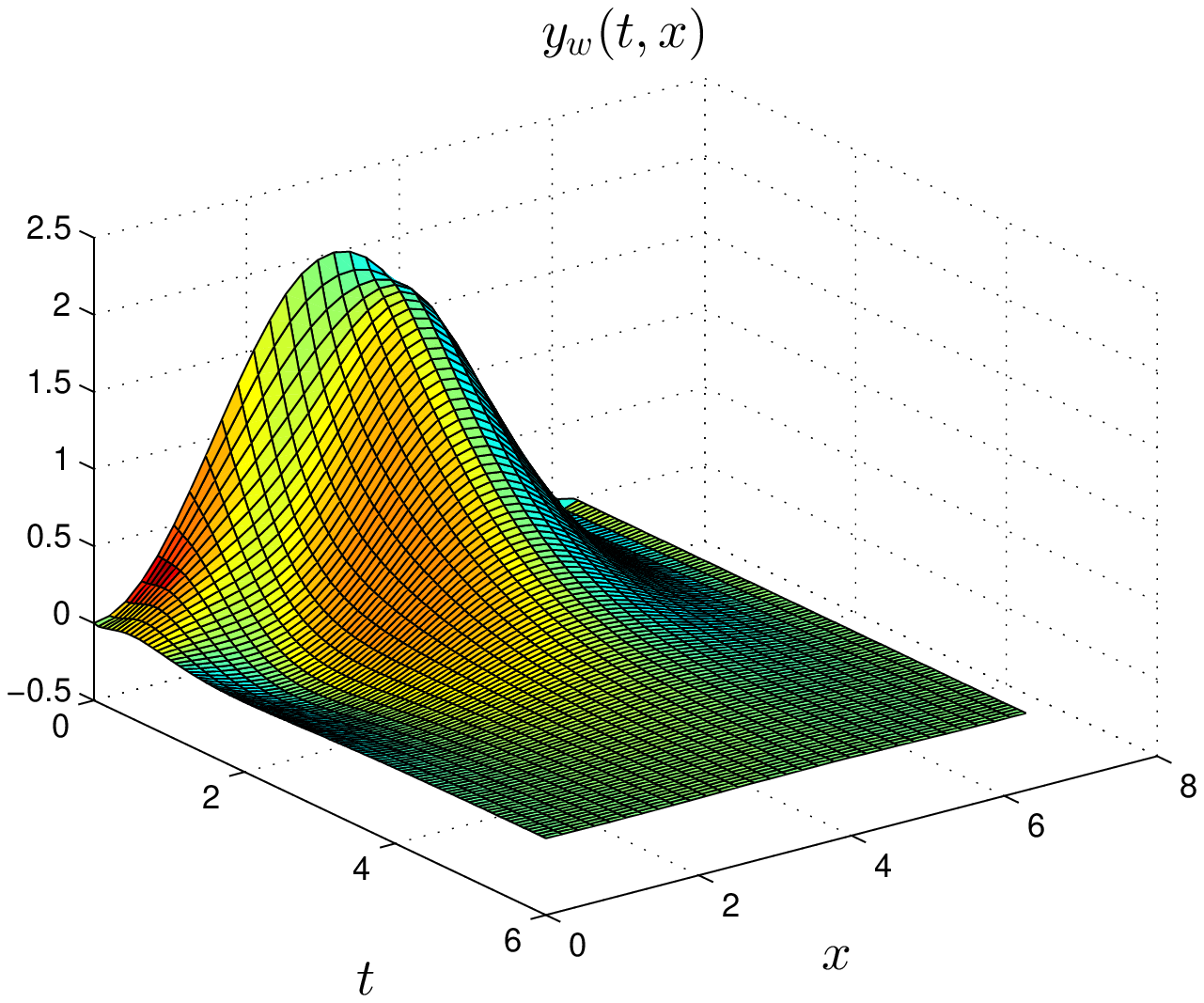}
      \caption{Solution $y_w(t,x)$ with a localized feedback law without saturation}
      \label{figure7}
   \end{minipage} \hfill
   \begin{minipage}[c]{.46\linewidth}
      \includegraphics[scale=0.6]{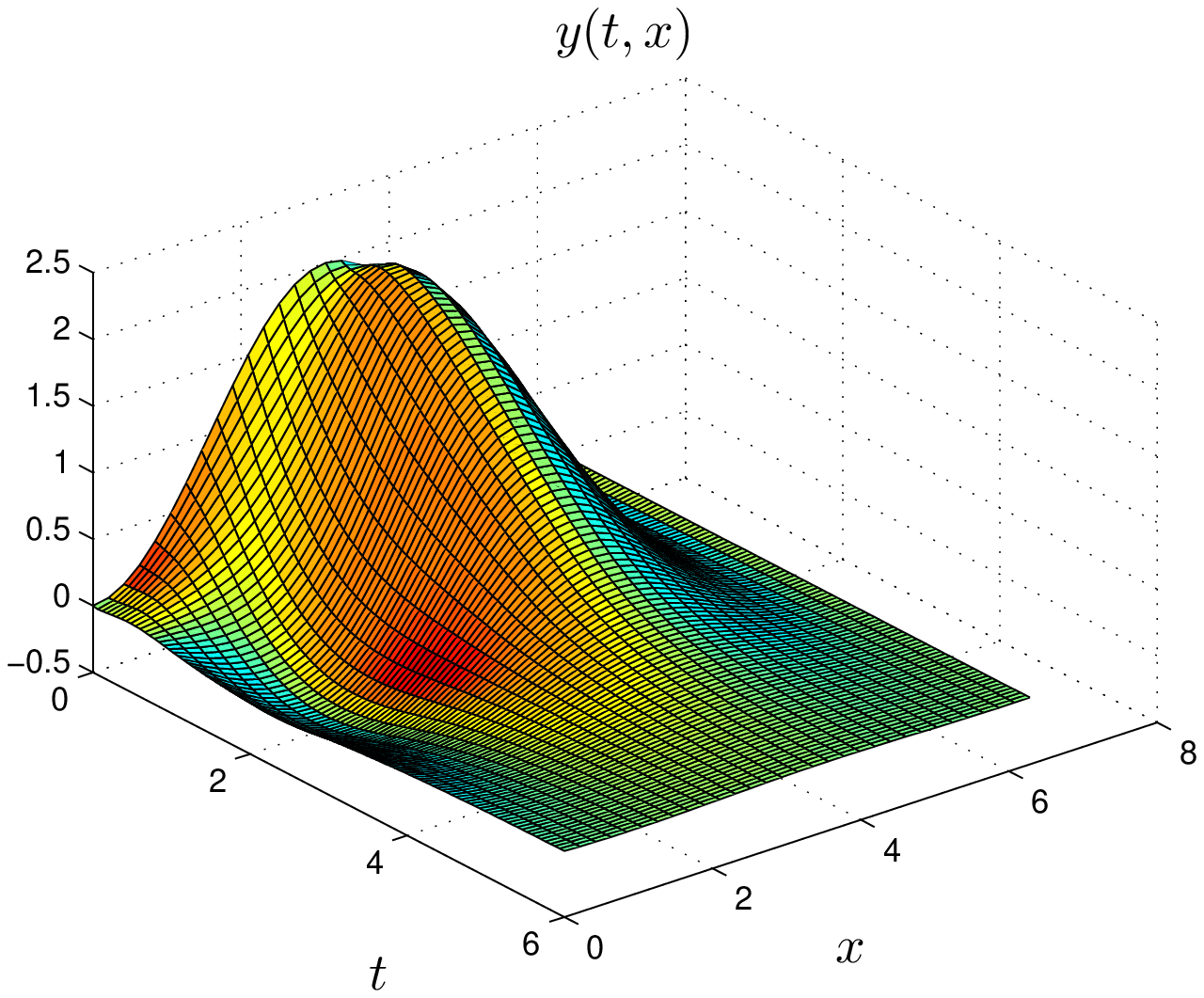}
      \caption{Solution $y(t,x)$ with a localized feedback law saturated; $u_0=0.5$}
      \label{figure8}
   \end{minipage}
\end{figure}

\begin{figure}[H]
   \begin{minipage}[c]{.46\linewidth}
      \includegraphics[scale=0.6]{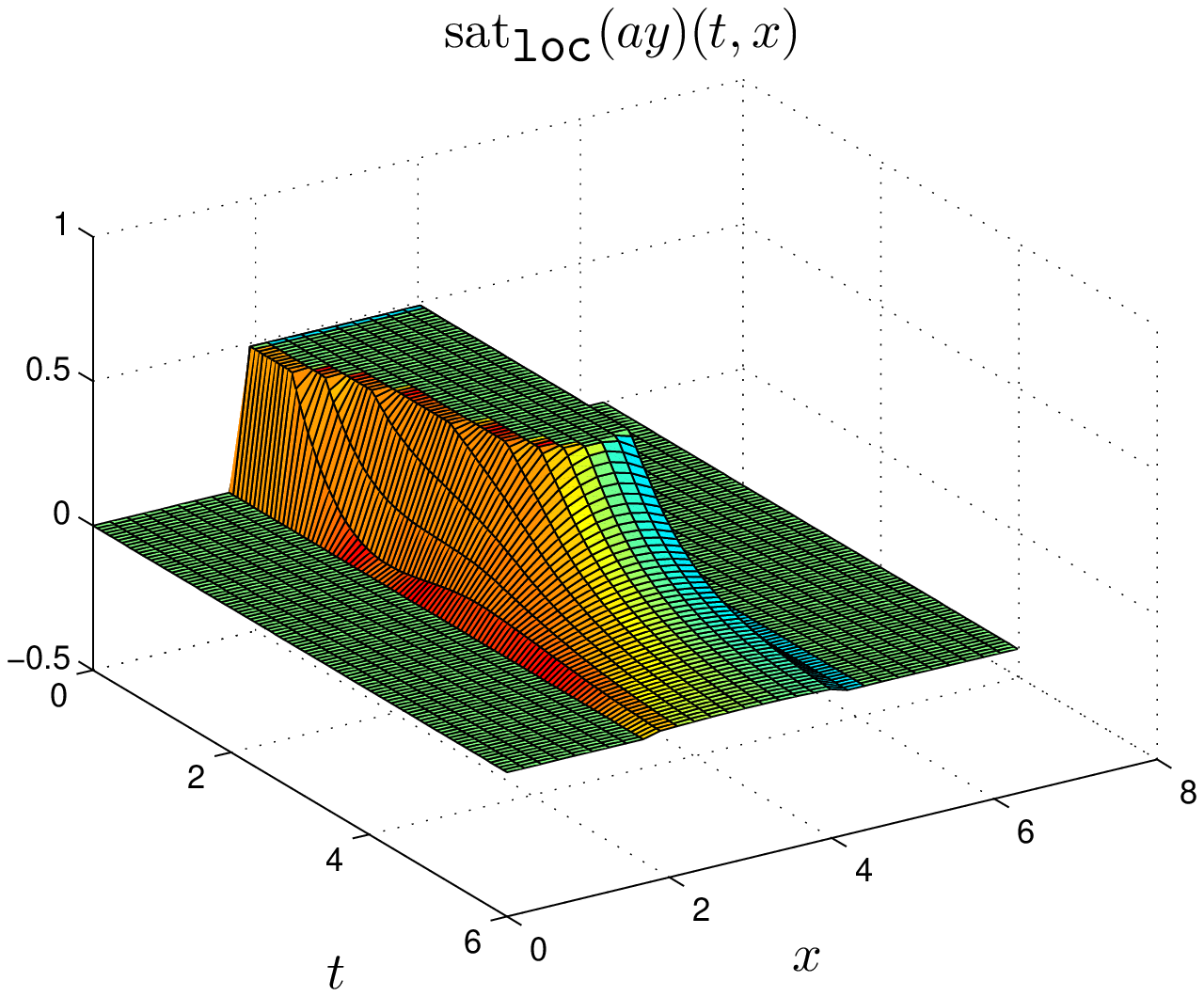}
      \caption{Control $f=\satl(a y)(t,x)$ where $\omega=\left[\frac{1}{3}L,\frac{2}{3}L\right]$, $u_0=0.5$}
      \label{figure9}
   \end{minipage} \hfill
   \begin{minipage}[c]{.46\linewidth}
      \includegraphics[scale=0.6]{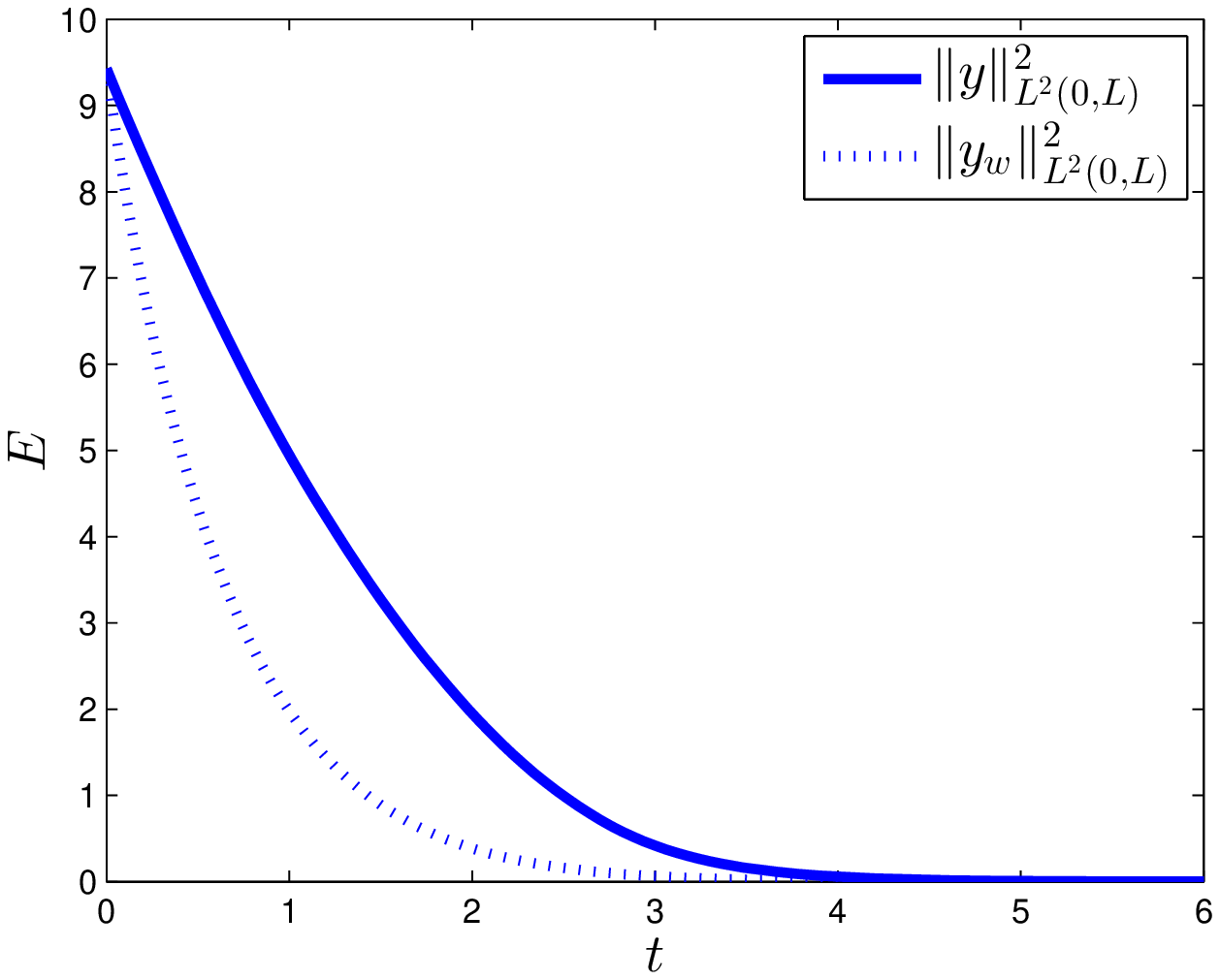}
      \caption{Blue: Time evolution of the energy function $\Vert y\Vert^2_{L^2(0,L)}$ with a saturation $u_0=0.5$, $a_0=1$ and $\omega=\left[\frac{1}{3}L,\frac{2}{3}L\right]$. Dotted line: Time evolution of the solution without saturation $y_w$ with $a_0=1$ and $\omega=\left[\frac{1}{3}L,\frac{2}{3}L\right]$.}
      \label{figure10}
   \end{minipage}
\end{figure}


\section{Conclusion}
\label{sec_conc}

In this paper, we have studied the well-posedness and the asymptotic stability of a Korteweg-de Vries equation with saturated distributed controls. The well-posedness issue has been tackled by using the Banach fixed-point theorem. The stability has been studied with two different methods: in the case where the control acts on all the domain saturated with $\sath$, we used a sector condition and Lyapunov theory for infinite dimensional systems; in the case where the control acts only on a part of the domain saturated with either $\sath$ or $\satl$, we argued by contradiction. We illustrate our results on some simulations, which show that the smaller is the saturation level, the slower is the convergence to zero.

To conclude, let us state some questions arising in this context:

\startmodif 1. Can a saturated localized damping stabilize in $H^3(0,L)$ a generalized Korteweg-de Vries equation, as done in the unsaturated case in \cite{rosier2006global} and \cite{linares2007exponential} ? \stopmodif

2. \startmodif Is it possible to saturate other damping terms, for instance the one suggested in \cite{perla-vasconcellos-zuazua} and used in \cite{massarolo} which dissipates the $H^{-1}$-norm in the unsaturated case? \stopmodif 

3. Some boundary controls have been already designed in \cite{cerpa_coron_backstepping}, \cite{coron2014local}, \cite{tang2013stabKdV} or \cite{cerpa2009rapid}. By saturating these controllers, are the corresponding equations still stable? 

4. Another constraint than the saturation can be considered. For instance the backlash studied in \cite{tarbouriech2014stability} or the quantization \cite{ferrante2015quantization}. 

5. Can we apply the same method for other nonlinear partial differential equations, for instance the Kuramoto-Sivashinsky equation \cite{cerpa2010cKSl, pato2014controlKS} ? \\



\textbf{Acknowledgements.} The authors would like to thank Lionel Rosier  for having attracted
our attention to the article \cite{rosier2006global} and for fruitful discussions.

\bibliographystyle{plain}
\bibliography{bibsm}
\end{document}